\documentclass[11pt,a4paper]{article}

\usepackage[left=2.3cm, top=2.3cm,bottom=2.3cm,right=2.3cm]{geometry}

\usepackage{mathtools,amssymb,amsthm,mathrsfs,calc,graphicx,stmaryrd,xcolor,dsfont,tikz,bbm,float,array,epsfig,hyperref,color}
\usepackage[numbers,square]{natbib}
\usepackage[british]{babel}
\usepackage{amsfonts}              

\numberwithin{equation}{section}
\numberwithin{figure}{section}

\newtheorem {theorem}{Theorem}[section]
\newtheorem {proposition}[theorem]{Proposition}
\newtheorem {lemma}[theorem]{Lemma}
\newtheorem {corollary}[theorem]{Corollary}

{\theoremstyle{definition}

\newtheorem {remark}[theorem]{Remark}

}

\newcommand{\Vol}{\operatorname{Vol}}

\def\ba{\begin{array}}
\def\ea{\end{array}}
\def\bea{\begin{eqnarray} \label}
\def\eea{\end{eqnarray}}
\def\be{\begin{equation} \label}
\def\ee{\end{equation}}
\def\bit{\begin{itemize}}
\def\eit{\end{itemize}}
\def\ben{\begin{enumerate}}
\def\een{\end{enumerate}}

\def\E{\mathbb{E}}

\def\N{\mathbb{N}}
\def\P{\mathbb{P}}

\def\R{\mathbb{R}}
\def\RRd1{\mathbb{R}^{d+1}}

\def\SSd{\mathbb{S}^d}

\def\bS{\mathbb{S}}

\def\cA{\mathcal{A}}

\def\cC{\mathcal{C}}

\def\cF{\mathcal{F}}

\def\cK{\mathcal{K}}

\def\cM{\mathcal{M}}
\def\cN{\mathcal{N}}
\def\cP{\mathcal{P}}

\def\cS{\mathcal{S}}

\def\dint{\textup{d}}

\newcommand{\eee}{{\rm e}}

\newcommand{\ind}{\mathbbm{1}}
\newcommand{\eps}{\varepsilon}
\newcommand{\pos}{\mathop{\mathrm{pos}}\nolimits}
\newcommand{\aff}{\mathop{\mathrm{aff}}\nolimits}
\newcommand{\conv}{\mathop{\mathrm{conv}}\nolimits}
\newcommand{\sconv}{\mathop{\mathrm{sconv}}\nolimits}
\newcommand{\dd}{{\rm d}}
\newcommand{\bsl}{\backslash}
\DeclareMathOperator{\relint}{relint}

\newcommand{\dist}{\mathop{\mathrm{dist}}\nolimits}

\newcommand{\Sd}{{\mathbb S}^{d-1}}

\begin{document}

\title{\bfseries Cones generated by random points on half-spheres\\ and convex hulls of Poisson point processes}

\author{Zakhar Kabluchko, Alexander Marynych, Daniel Temesvari and Christoph Th\"ale}

\date{}

\maketitle

\begin{abstract}
Let $U_1,U_2,\ldots$ be random points sampled uniformly and independently from the $d$-dim\-en\-sional upper half-sphere.  We show that, as $n\to\infty$, the $f$-vector of the $(d+1)$-dimensional convex cone $C_n$ generated by $U_1,\ldots,U_n$  weakly converges to a certain limiting random vector, without any normalization. We also show convergence of all moments of the $f$-vector of $C_n$ and identify the limiting constants for the expectations. We prove that the expected Grassmann angles of $C_n$ can be expressed through the expected $f$-vector. This yields convergence of expected Grassmann angles and conic intrinsic volumes and answers thereby a question of B\'ar\'any, Hug, Reitzner and Schneider [Random points in halfspheres, \textit{Rand. Struct. Alg.}, 2017]. Our approach is based on the observation that the random cone $C_n$ weakly converges, after a suitable rescaling, to a random cone whose intersection with the tangent hyperplane of the half-sphere at its north pole is the convex hull of the Poisson point process with power-law intensity function proportional to $\|x\|^{-(d+\gamma)}$, where $\gamma=1$. We compute the expected number of facets, the expected intrinsic volumes and the expected $T$-functional of this random convex hull for arbitrary $\gamma>0$.
\end{abstract}

\tableofcontents

\section{Introduction}

The study of random convex hulls has a long tradition in convex and stochastic geometry; see Chapter 8 in~\cite{SW08} as well as~\cite{hug_survey,MCR10,schneider_polytopes} for overviews. Motivated by Sylvester's four-point problem, the modern development started with the works of R\'enyi and Sulanke~\cite{renyi_sulanke1,renyi_sulanke2} on random polygons in the plane that are generated as convex hulls of random points uniformly distributed in a fixed convex set. Random polytopes of this type in general space dimensions $d\geq 1$ were studied intensively, for example by B\'ar\'any, Reitzner, Sch\"utt~\cite{Barany89,ReitznerCombinatorialStructure,SchuettRandomPolytopes}, to name just a few. One of the functionals that attracted particular interest was the so-called \textit{$f$-vector}, that is, the vector whose $k$-th entry is the number of $k$-dimensional faces, $k\in\{0,1,\ldots d-1\}$. In particular, if a random polytope $K_n$ is generated by $n\geq d+1$ independent random points that are uniformly distributed in a convex body $K\subset\R^d$ with (sufficiently) smooth boundary, the expected number $\E f_k(K_n)$ of $k$-dimensional faces of $K_n$ asymptotically behaves like
$$
\E f_k(K_n) \sim c_{k,d}\Omega(K)n^{\frac{d-1}{d+1}},
$$
as $n\to\infty$. Here, $c_{k,d}\in(0,\infty)$ is a constant that only depends on $k$ and $d$, and $\Omega(K)$ is the affine surface area of $K$; see~\cite{ReitznerCombinatorialStructure}. On the other hand, if $K$ itself is a polytope, the expected number of $k$-dimensional faces of $K_n$ grows like
$$
\E f_k(K_n) \sim c_{k,d}'T(K)(\log n)^{d-1},
$$
as $n\to\infty$, with a different constant $c'_{k,d}\in(0,\infty)$ and with $T(K)$ being the number of towers of the polytope $K$; see again~\cite{ReitznerCombinatorialStructure}.

Recently, B\'ar\'any, Hug, Reitzner and Schneider~\cite{barany_etal} investigated the $f$-vector, the spherical volume and some other quantities for the spherical convex hull of $n$ uniformly distributed random points on the $d$-dimensional upper half-sphere. Among other results, they showed that the expected number of facets (i.e.\ $(d-1)$-dimensional faces) and the expected number of vertices and edges of such spherical random polytopes tend to finite constants, as $n\to\infty$. This surprising result is the starting point for our work in which we consider the $(d+1)$-dimensional random convex cone generated by such random convex hulls; see Figure~\ref{fig}. Our first main result (Theorem~\ref{theo:conv_of_polytopes}) is a weak limit theorem for the sections of these random cones with the tangent hyperplane of the half-sphere at its north pole. We shall identify the limiting random polytope as the convex hull of a Poisson point process in the tangent hyperplane with a power-law intensity function. This in turn leads to limit theorems for the whole $f$-vector (Theorem~\ref{theo:conv_of_faces} and Theorem \ref{theo:ConvergenceOfExpectations}) and the volume (Theorem~\ref{theo:conv_of_angle}) of the spherical convex hull on a half-sphere, which complements the findings in~\cite{barany_etal}. In addition, our weak limit theorem allows us to describe the expectation asymptotics of the conic intrinsic volumes (in fact, all three versions of them) of the induced random cone. This solves in an extended form a conjecture posed by B\'ar\'any, Hug, Reitzner and Schneider; see Section 9 in \cite{barany_etal}.  We also study separately the expected so-called $T$-functional of the convex hull of a general class of Poisson point processes in $\R^d$ with a power-law intensity function $\|x\|^{-(d+\gamma)}$; see Theorem~\ref{theo:T_expect}. Here, $\gamma>0$ is a parameter and $\|x\|$ is the Euclidean norm of $x$. In particular, we compute explicitly the expected volume (and, more generally, expected intrinsic volumes) and the expected number of facets of this random polytope, thus generalizing a two-dimensional result of~\citet{Davis+Mulrow+Resnick:1987}.

\medbreak

The paper is structured as follows. In Section~\ref{subsec:HalfSphere} we first rephrase the relevant results from~\cite{barany_etal} and introduce the random convex cones for which various limit theorems are presented in Sections~\ref{subsec:WeakLimiTheorems} and~\ref{subsec:conic_intrinsic}. Convex hulls of Poisson point processes with a power-law intensity function are the content of Section~\ref{subsec:ConvexHullPPP}. In order to keep the paper reasonably self-contained we have collected some background material needed in our arguments in Section~\ref{sec:BackgroundMaterial}. The proofs of our main results are contained in Sections~\ref{sec:ProofsWeakLimitTheorems}, \ref{proofs:conic} and~\ref{sec:ProofsPPP}, while Section~\ref{sec:AuxLemmas} collects some auxiliary lemmas.

\section{Main results}

\subsection{Convex hulls on the half-sphere}\label{subsec:HalfSphere}

We fix a dimension $d\geq 1$ and let $U_1,U_2,\ldots$ be independent random points distributed according to the uniform distribution on the $d$-dimensional \textit{upper half-sphere}
$$
\mathbb S_+^{d} := \{(x_0,x_1,\ldots,x_d) \in \R^{d+1} \colon x_0^2+x_1^2+\ldots+x_d^2=1,\,x_0\geq 0\}.
$$
We are interested in the random convex cone in $\R^{d+1}$ defined as the \textit{positive hull} of $U_1,\ldots,U_n$, $n\geq d+1$, that is
$$
C_n = \pos \{U_1,\ldots,U_n\} := \{\alpha_1 U_1+\ldots + \alpha_n U_n \colon \alpha_1,\ldots, \alpha_n\geq 0\};
$$
see Figure~\ref{fig}. As already discussed in the previous section, the random cone, or, more precisely, the random spherical polytope $C_n \cap \mathbb S^{d}_+$, has been studied by~\citet{barany_etal}. Some of their results concern the expected $f$-vector of $C_n$, that is, the expected number $\E f_k(C_n)$ of $k$-dimensional faces of $C_n$, $k\in\{1,\ldots,d\}$. The $f$-vector of the cone $C_n$ is related to the $f$-vector of the spherical polytope $C_n \cap \mathbb S^{d}_+$ by $f_{k}(C_n) = f_{k-1} (C_n \cap \mathbb S^{d}_+)$. For our purposes, it is more convenient to work with cones rather than with spherical polytopes.  By~\cite[Theorem 3.1]{barany_etal} the expected number of facets $\E f_{d}(C_n)$ of $C_n$  is explicitly given by
\begin{equation}\label{eq:barany_fd}
\E f_{d} (C_n) = \frac {2 \omega_d}{\omega_{d+1}} \binom nd \int_0^\pi \left(1-\frac \alpha \pi\right)^{n-d} \sin^{d-1}\alpha \,\dd \alpha.
\end{equation}
Moreover, it has been shown in~\cite[Theorem~3.1]{barany_etal} that
\begin{equation}\label{eq:barany_fd_asympt}
\lim_{n\to\infty} \E f_{d} (C_n) = 2^{-d}d! \kappa_d^2.
\end{equation}
Here and below, $\kappa_d$ denotes the volume of the $d$-dimensional unit ball, whereas $\omega_d$ is the $(d-1)$-dimensional Hausdorff measure (surface area) of the unit sphere $\mathbb S^{d-1}\subset \R^d$, that is,
$$
\kappa_d = \frac{\pi^{d/2}}{\Gamma(\frac d2 + 1)}
\qquad\text{and}\qquad
\omega_d = d \kappa_d =\frac{2\pi^{\frac{d}{2}}}{\Gamma\left(\frac{d}{2}\right)}.
$$
Regarding the expected number of one-dimensional faces of $C_n$ (or, equivalently, vertices of $C_n \cap \mathbb S^d_+$), \cite[Theorem 7.1]{barany_etal} says that
\begin{equation}\label{eq:barany_f1}
\lim_{n\to\infty} \E f_1 (C_n) = C(d) \pi^{d+1} \left(\frac 2 {\omega_{d+1}}\right)^{d+1} \omega_d
\end{equation}
for a certain constant $C(d)$ given in form of a multiple integral; see~\cite[Equation~(22)]{barany_etal}.  Let us also mention that cones generated by random points with uniform distribution on the whole sphere $\bS^d$ were studied by \cite{cover_efron} and~\cite{HS15}.

\subsection{Weak convergence of the random cone and its consequences}\label{subsec:WeakLimiTheorems}

\subsubsection{The weak convergence theorem}\label{subsec:weak_conv_thm}

\begin{figure}[t]
\begin{center}
\includegraphics[width=0.6\textwidth]{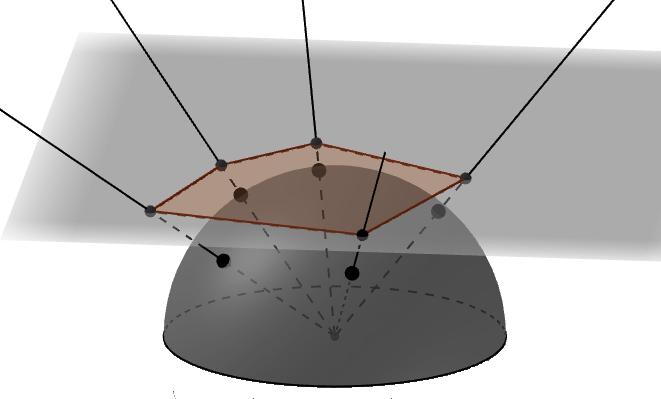}
\end{center}
\caption{Illustration of random points on the half-sphere $\mathbb{S}_+^d$, the cone $C_n\subset\R^{d+1}$ and the intersection $C_n\cap H_1$ for $d=2$ and $n=5$.}
\label{fig}
\end{figure}

In what follows, we shall present a weak limit theorem for the random cone $C_n$. It is clear that, for large $n$, the cone $C_n$ is close to the half-space $\{x_0>0\}$, so that in order to obtain a non-trivial limit for $C_n$ we need an appropriate rescaling. This is achieved by the linear operator $T_n:\R^{d+1} \to \R^{d+1}$ defined by
$$
T_n(x_0,x_1,\ldots,x_d) := (n x_0, x_1,\ldots, x_d).
$$
Let $H_1$ be the hyperplane $\{x_0=1\}$ in $\R^{d+1}$. Note that $H_1$ is tangent to the half-sphere $\mathbb S^d_+$ at its north pole.  Let $e_0$ be the unit vector $(1,0,\ldots,0)\in\R^{d+1}$ pointing to the north pole. We shall prove that the random convex polytope $(T_n C_n \cap H_1) - e_0$, which can be viewed as the ``horizontal'' section of the cone $T_nC_n$, converges in distribution on the space of compact convex subsets of $H_1-e_0$ that we identify with $\R^d$; see Section~\ref{sec:BackgroundMaterial} below for some background material on this notion of convergence.

To describe the limit, take some $\gamma > 0$, $c > 0$,  and let $\Pi_{d,\gamma}(c)$ be a Poisson point process on $\R^d \bsl \{0\}$ whose intensity measure is absolutely continuous with respect to the Lebesgue measure and whose density function is given by
\begin{equation}\label{eq:density}
x\mapsto \frac c {\omega_{d+\gamma}} \frac{1}{\|x\|^{d+\gamma}},\quad x\in \R^d \bsl \{0\},
\end{equation}
where $\|x\|$ is the Euclidean norm of $x$; see Figure~\ref{fig_PPP}. Again, we refer to Section~\ref{sec:BackgroundMaterial} for background material concerning Poisson point processes. Note that the number of points of $\Pi_{d,\gamma}(c)$ outside any ball centered at the origin having strictly positive radius is almost surely finite (because the intensity is integrable near $\infty$), while the number of points inside any such ball is infinite with probability one (because the integral of the intensity over such balls diverges).  We denote by $\conv \Pi_{d,\gamma}(c)$ the convex hull of all points of $\Pi_{d,\gamma}(c)$. Even though $\Pi_{d,\gamma}(c)$ almost surely consists of infinitely many points, the random convex set $\conv \Pi_{d,\gamma}(c)$ turns out to be almost surely a polytope; see Corollary~\ref{cor:polytope} below. The next theorem identifies the weak limit of the rescaled random polytopes $(T_nC_n\cap H_1)-e_0$ in terms of a Poisson point process of the type just discussed.

\begin{theorem}\label{theo:conv_of_polytopes}
As $n\to\infty$, the random polytopes $(T_n C_n \cap H_1) - e_0$  converge in distribution to $\conv \Pi_{d,1}(2)$ on the space of compact convex subsets of $\R^d$ endowed with the Hausdorff metric.
\end{theorem}

Let us briefly explain the idea behind Theorem~\ref{theo:conv_of_polytopes}. Define the map $\mathcal{P}:\mathbb{S}^{d}_{+}\cap \{x_0>0\}\to \mathbb{R}^{d}$ by the equality
\begin{align}\label{eq:DefMappingP}
\mathcal{P}(x_0,x_1,\ldots,x_d)=\bigg({x_1\over x_0},\ldots,{x_d\over x_0}\bigg).
\end{align}
The rays in directions $U_1,\ldots,U_n$ intersect $H_1$ at the points $(1,\mathcal P(U_1)),\ldots,(1,\mathcal P(U_n))$. Therefore, the polytope $C_n\cap H_1-e_0$ is the convex hull of $\mathcal P(U_1),\ldots,\mathcal P(U_n)$. The next proposition describes the density according to which these points are distributed. The result is a consequence of \cite[Proposition 4.2]{BesauWerner} and, in a more general set-up, has been proved in the argument of \cite[Theorem 7]{bonnet_etal}.

\begin{proposition}\label{prop:cauchy}
Let $(\xi_0,\ldots,\xi_d)$ be a random vector distributed uniformly on the half-sphere $\mathbb S^d_+$. Then, the vector $\mathcal P(\xi_0,\xi_1,\ldots,\xi_d) := (\xi_1/\xi_0,\ldots,\xi_d/\xi_0)$ has the following generalized Cauchy density
$$
x\mapsto {2\over\omega_{d+1}}\frac{1}{(1+\|x\|^{2})^{\frac{d+1}2}}, \qquad x\in\R^d.
$$
\end{proposition}

Note that this density belongs to the class of beta$^{\prime}$-distributions. Convex hulls of samples from these distributions were studied in~\cite{beta_polytopes}. In particular, the formula for the number of facets of this convex hull obtained in~\cite[Proposition~3.16]{beta_polytopes} contains~\eqref{eq:barany_fd} as a special case. Let us turn to the large $n$ asymptotics. Since the above density is regularly varying at $\infty$, see Lemma~\ref{lem:lemma_reg_var} in Section~\ref{sec:AuxLemmas}, standard methods from extreme-value theory imply that the point process formed by the points $\mathcal P(U_1)/n, \ldots, \mathcal P(U_n)/n$ converges weakly to the Poisson point process $\Pi_{d,1}(2)$ in the space of locally-finite integer measures on  $\R^d \bsl \{0\}$ endowed with the vague topology. Using the continuous mapping theorem, we shall argue that the convex hull of $\mathcal P(U_1)/n, \ldots, \mathcal P(U_n)/n$ converges weakly to the convex hull of the Poisson point process, thus proving Theorem~\ref{theo:conv_of_polytopes}.

\begin{figure}[t]
\begin{center}
\includegraphics[width=0.45\textwidth]{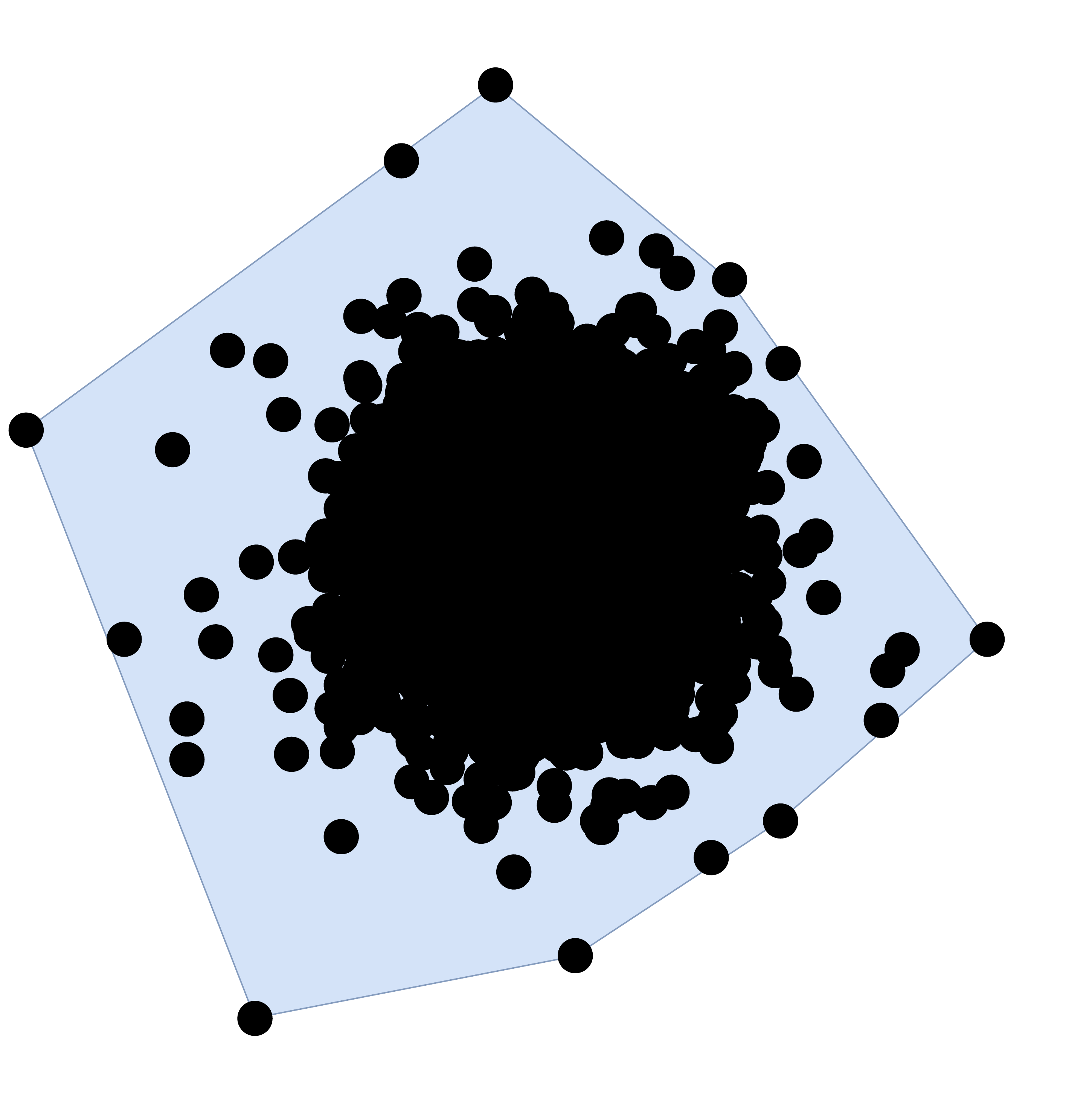}
\includegraphics[width=0.45\textwidth]{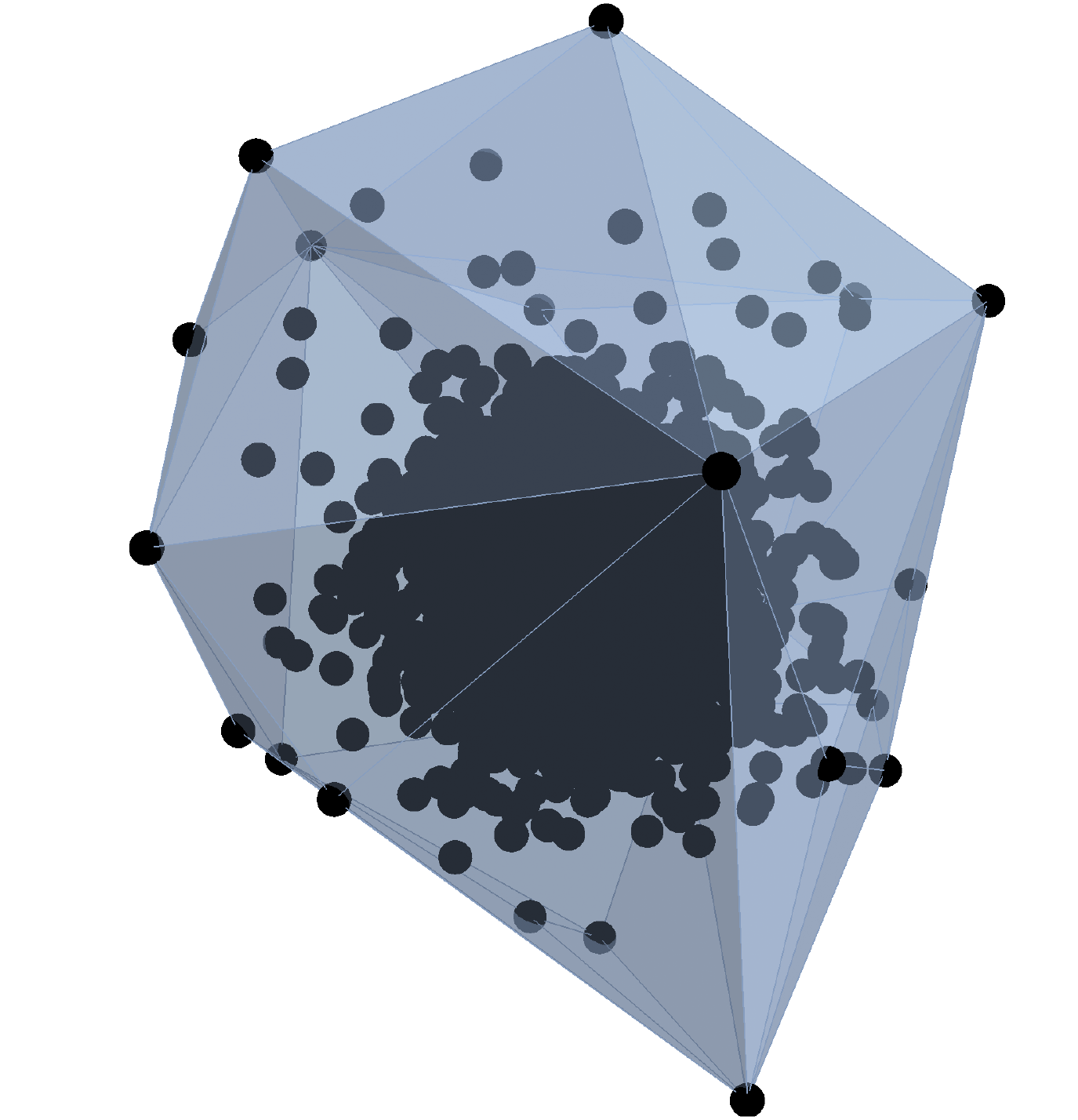}
\end{center}
\caption{Convex hull of the Poisson point process with intensity $\|x\|^{-7}$ in $d=2$ (left) and $d=3$ (right) dimensions.}
\label{fig_PPP}
\end{figure}

\subsubsection{Convergence of the \texorpdfstring{$f$}{f}-vector}

With the help of the continuous mapping theorem we shall now derive a number of consequences of Theorem~\ref{theo:conv_of_polytopes}.  For a Euclidean or spherical $d$-dimensional polytope $P$, we denote by $f_k(P)$ the number of $k$-dimensional faces of $P$, where $k\in\{0,\ldots,d-1\}$.  The collection ${\bf f}(P):=(f_0(P), \ldots, f_{d-1}(P))$ is the $f$-vector of $P$.  From Theorem~\ref{theo:conv_of_polytopes} we shall derive the following result on the distributional convergence of the $f$-vector of the random spherical polytope $C_n \cap \mathbb S^d_+$. We remind the reader that $f_{k}(C_n \cap \mathbb S^d_+)= f_{k+1}(C_n)$.

\begin{theorem}\label{theo:conv_of_faces}
As $n\to\infty$, we have that
$$
{\bf f}(C_n \cap \mathbb S^d_+) \overset{{\rm d}}{\longrightarrow} {\bf f}(\conv \Pi_{d,1}(2)),
$$
where $\overset{{\rm d}}{\longrightarrow}$ denotes convergence in distribution.
\end{theorem}

We shall argue also that the \textit{expected} $f$-vector of the spherical random polytope $C_n \cap \mathbb S^d_+$ converges to that of $\conv \Pi_{d,1}(2)$. Even more generally, we shall prove the convergence of moments of all orders. This generalizes the results from \cite{barany_etal} discussed above and answers -- in an extended form -- a question raised in~\cite[Section~9]{barany_etal}. Let us write $\aff\{x_1,\ldots,x_k\}$ for the affine hull of the points $x_1,\ldots,x_k$.

\begin{theorem}\label{theo:ConvergenceOfExpectations}
For every $k\in \{1,\ldots,d\}$ and every $m\in\N$ we have
$$
\lim_{n\to\infty} \E f_k^m (C_n) = \lim_{n\to\infty} \E f_{k-1}^m(C_n \cap \mathbb S^d_+) =  \E f_{k-1}^m (\conv \Pi_{d,1}(2)).
$$
For $m=1$ the limits of the expectations are
$$
\lim_{n\to\infty} \E f_k (C_n) = \lim_{n\to\infty} \E f_{k-1}(C_n \cap \mathbb S^d_+) = \E f_{k-1} (\conv \Pi_{d,1}(2)) = \frac {2}{k!} B_{k,d},
$$
where  $B_{1,d},\ldots,B_{d,d}$ are constants given by
\begin{align}\label{eq:DefBk}
B_{k,d} = {1\over 2}\Big({2\over\omega_{d+1}}\Big)^{k}\int_{(\R^d)^{k}}\P(\conv \Pi_{d,1}(2)\cap\aff\{x_1,\ldots,x_{k}\} = \varnothing)\prod_{i=1}^{k}{\dd x_i\over\|x_i\|^{d+1}}<\infty.
\end{align}
\end{theorem}
\begin{remark}
We shall prove in Section~\ref{subsec:proof_exp_T_funct} that
\begin{equation}\label{eq:ConstantBd}
B_{d,d}=(2\pi)^{d-1}\Gamma\Big({d+1\over 2}\Big)^2.
\end{equation}
Together with Theorem~\ref{theo:ConvergenceOfExpectations} and Legendre's duplication formula this recovers Equation~(4) of~\citet{barany_etal} who proved that  $\lim_{n\to\infty} \E f_d(C_n)= 2^{-d} d! \kappa_d^2$. In Proposition~\ref{prop:B_2_d}, we shall compute the value of $B_{2,d}$, yielding the formula
$$
\lim_{n\to\infty} \E f_2(C_n) = B_{2,d} = \frac 12 \binom{d+1}{3} \pi^2.
$$
\end{remark}

\subsubsection{Convergence of the solid angle}

The next theorem deals with the solid angle of $C_n$. Let $\bar \sigma$ be the $d$-dimensional spherical Lebesgue measure on the unit sphere $\mathbb{S}^d\subset \R^{d+1}$ normalized such that $\bar \sigma(\mathbb{S}^d)=1$. The \textit{solid angle} $\alpha(C_n)$ of the convex cone $C_n$ is defined by
$$
\alpha(C_n):= \bar \sigma(C_n\cap \mathbb{S}^d).
$$
Clearly, we have that $\alpha(C_n)$ almost surely converges to $1/2$, as $n\to\infty$. Theorem 7.1 in~\cite{barany_etal} provides a more delicate asymptotic result, namely
\begin{align}\label{eq:ConstangCd2}
\E \left(\frac 12 - \alpha(C_n)\right) =
C(d) \pi^{d+1} \left(\frac 2 {\omega_{d+1}}\right)^{d+1} \frac{\omega_d}{\omega_{d+1}} \frac 1 n  + O(n^{-2}),
\end{align}
as $n\to\infty$, where $C(d)$ is the same constant as in~\eqref{eq:barany_f1}. The next theorem is a distributional counterpart to this formula.

\begin{theorem}\label{theo:conv_of_angle}
As $n\to\infty$, we have that
$$
n\left(\frac 12  - \alpha(C_n) \right) \overset{{\rm d}}{\longrightarrow}  \frac{1}{\omega_{d+1}} \int_{\R^d \backslash \conv\Pi_{d,1}(2)}  \frac{\dd x}{\|x\|^{d+1}}.
$$
\end{theorem}

\subsection{Conic intrinsic volumes}\label{subsec:conic_intrinsic}
Next we consider the so-called conic intrinsic volumes of $C_n$ or, equivalently, the spherical intrinsic volumes of $C_n\cap\mathbb{S}_+^d$. In contrast to the classical intrinsic volumes in $\R^d$ there exist several notions of conic intrinsic volumes in the literature; see \cite{AmelunxenLotzDCG17,ALMT14} and, for equivalent formulations in the spherical setting, \cite{glasauer_phd} and~\cite[Section 6.5]{SW08}. If $C\subset\R^{d+1}$ is a polyhedral convex cone and $x\in\R^{d+1}$ we let $\Pi_C(x)$ be the metric projection of $x$ onto $C$, that is $\Pi_C(x)$ is the uniquely determined point $y\in C$ for which the squared Euclidean distance $\|x-y\|^2$ is minimal. If $g$ is a standard Gaussian random vector in $\R^{d+1}$ and $F\subseteq C$ is a face of $C$ with relative interior denoted by $\relint(F)$, we put $v_F:=\P(\Pi_C(g)\in\relint(F))$ and
$$
v_k(C) := \sum_{F\in\cF_k(C)}v_F,\qquad k\in\{0,1,\ldots,d+1\},
$$
where $\mathcal F_k(C)$ is the set of all $k$-dimensional faces of $C$. For convenience also define $v_k(C):=0$ for $k>d+1$. This is the $k$th \textit{conic intrinsic volume} of $C$. We notice that the conic intrinsic volumes of the upper halfspace $H_{\rm up}:=\{x=(x_0,\ldots,x_{d})\in\R^{d+1}:x_{0}\geq 0\}$ are given by $v_k(H_{\rm up})=0$ if $k\in\{0,1,\ldots,d-1\}$ and $v_d(H_{\rm up})=v_{d+1}(H_{\rm up})=1/2$. If $C$ is a $k$-dimensional linear subspace, then $v_{k}(C)=1$, while all other conic intrinsic volumes vanish. Henceforth, we shall always exclude the case when $C$ is linear subspace (since formulae~\eqref{eq:gauss_bonnet} and~\eqref{eq:h_def} below are not valid in this case).  One important property of the conic intrinsic volumes is the \textit{Gauss--Bonnet formula} \cite[p.~28, Equation~(5.3)]{ALMT14}
\begin{equation}\label{eq:gauss_bonnet}
v_0(C)+ v_2(C) + \ldots = v_1(C)+ v_3(C) +\ldots = \frac 12.
\end{equation}

Next let us define the so-called \textit{Grassmann angles} of a polyhedral convex cone $C$.
If $G(d+1,d+1-k)$ with $k\in\{0,1,\ldots,d+1\}$ is the Grassmannian of all $(d+1-k)$-dimensional linear subspaces with the unique Haar probability measure $\nu_{d+1-k}$ (see Section~\ref{sec:BackgroundMaterial} below) and $L\in G(d+1,d+1-k)$ is a random subspace distributed according to $\nu_{d+1-k}$, we can define
\begin{equation}\label{eq:h_def}
h_{k+1}(C) := {1\over 2}\P(C\cap L\neq\{0\}),\qquad k\in\{0,1,\ldots,d\}.
\end{equation}
This is the $(k+1)$st Grassmann angle of $C$ that has been introduced by Gr\"unbaum \cite{GruenbaumGA}.   In particular, the $(d+1)$st Grassmann angle $h_{d+1}(C)$ coincides with the solid angle $\alpha(C)$ studied above. Note also that all Grassmann angles $h_1,\ldots,h_{d+1}$ of the upper halfspace $H_{\rm up}$ are equal to $1/2$.   The \textit{conic Crofton formula}~\cite[Equation (2.10)]{AmelunxenLotzDCG17} states that the conic intrinsic volumes and the Grassmann angles are related by
\begin{equation}\label{eq:ConicalIntVolGrassmannAngle}
h_{k+1}(C) = \sum_{\substack{i\geq 1\\ i\text{ odd}}}v_{k+i}(C).
\end{equation}
In the terminology of~\cite{ALMT14}, the above sums (which are in fact finite) are called the \textit{half-tail functionals}.
For every cone $C$ we have $h_1(C)=1/2$ and we put $h_0(C) = 1/2$, $h_{d+2}(C)=h_{d+3}(C)=\ldots = 0$ in order to be consistent with~\eqref{eq:gauss_bonnet}.

Finally, we may consider the \textit{conic mean projection volumes} defined for $k\in\{0,1,\ldots,d\}$ by
$$
w_{k+1}(C) := \frac 1{\kappa_{k+1}}\int_{G(d+1,k+1)}\Vol_{k+1}(P_L(C)\cap\mathbb{B}^{d+1})\,\nu_{k+1}(\dd L),
$$
where $\Vol_{k+1}$ stands for the Lebesgue measure in $L\in G(d+1,k+1)$, $P_L$ for the orthogonal projection onto $L$ and $\mathbb{B}^{d+1}$ for the $(d+1)$-dimensional unit ball. The conic mean projection volumes are related to the conic intrinsic volumes via what may be called the \textit{conic Kubota formula}
\begin{align}\label{eq:ConicalMeanProjectionIntVol}
w_{k+1}(C) = \sum_{i=k+1}^{d+1} v_i(C) = h_{k+1}(C) + h_{k+2}(C),
\end{align}
see Lemma \ref{lem:ConicalMeanProjection}. Thus, the conic mean projection volumes coincide with the \textit{tail functionals} in the language of~\cite{ALMT14}. For the half-space $H_{\rm up}$ we have $w_{1}(H_{\rm up}) = \ldots = w_{d}(H_{\rm up}) = 1$ and $w_{d+1}(H_{\rm up}) = 1/2$.

The next result relates the expected Grassmann angles of the random cone $C_n$ to its expected $f$-vector.
\begin{theorem}\label{theo:buchta}
For all $k\in\{1,\ldots,d\}$ we have
\begin{align*}
2 \binom {n+d+1-k}{d+1-k} \Big(\frac 12 - \E h_{k+1}(C_n)\Big) = \E  f_{d+1-k} (C_{n+d+1-k}).
\end{align*}
\end{theorem}
The above formula should be compared to the well-known Efron identity~\cite{efron} that states that for random points $Q_1,Q_2,\ldots$ sampled uniformly and independently from a convex body $K\subset \R^d$ and all $n\geq d+1$ we have
$$
\frac{\E \text{Vol}_d \conv\{Q_1,\ldots,Q_n\}}{\text{Vol}_d(K)} = 1 - \frac{\E f_0(\conv\{Q_1,\ldots,Q_{n+1}\})}{n+1}.
$$
Buchta~\cite{buchta} obtained an analogue of this identity for higher moments of the volume, but no identity relating the expected $f$-vector of random polytopes to their intrinsic volumes is known in the Euclidean case, to the best of our knowledge (however, we refer to \cite{HoermannHugReitznerThaele,SchneiderWeightedFaces} for results in this direction for the zero cells of Poisson hyperplane tessellations).

Our next result identifies asymptotically the expected conic intrinsic volumes, the Grassmann angles and the conic mean projection volumes of the random cones $C_n$. Note that this completely settles in an extended form the conjecture of B\'ar\'any et al.\ stated in \cite[Section 9]{barany_etal}.

\begin{theorem}\label{thm:ConicalIntrinsicVolumes}
For every $k\in\{0,1,\ldots,d\}$ we have
\begin{equation}\label{eq:conic_asympt1}
\lim_{n\to\infty}n^{d+1-k}\Big({1\over 2}-\E h_{k+1}(C_n)\Big) = B_{d+1-k,d},
\end{equation}
where $B_{1,d},\ldots, B_{d,d}$ are given by~\eqref{eq:DefBk}, and $B_{d+1,d}=0$.
Moreover, for all $\ell,r\in\{0,1,\ldots,d-1\}$ we have
\begin{align}
& \lim_{n\to\infty}n^{d-\ell}\,\E v_\ell(C_n) = B_{d-\ell,d}, \label{eq:conic_asympt2}\\
& \lim_{n\to\infty}n^{d-r} \big(1 - \E w_{r+1}(C_n)\big) = B_{d-r,d}. \label{eq:conic_asympt3}
\end{align}

\end{theorem}

\begin{remark}
Note that $v_d(C_n) = h_d(C_n) \to 1/2$ and $v_{d+1}(C_n) = h_{d+1}(C_n) \to 1/2$, as $n\to\infty$, hence we have restricted ourselves to the conic intrinsic volumes $v_l(C_n)$ of orders $\ell \in \{0,\ldots,d-1\}$ in~\eqref{eq:conic_asympt2}. Similarly, $w_{d+1}(C_n) = h_{d+1}(C_n)$, hence we omitted the case $r=d$ in~\eqref{eq:conic_asympt3}.
\end{remark}

\begin{proposition}\label{prop:B_2_d}
For all $d\geq 2$ we have
$$
B_{2,d} = \frac 12 \binom {d+1}{3} \pi^2.
$$
\end{proposition}
\begin{proof}
For the expected surface area (i.e.\ $(d-1)$-dimensional Hausdorff measure) of the spherical polytope $C_n \cap \bS^{d}$, \citet{barany_etal} showed in their Theorem~5.1 that
$$
\E S(C_n \cap \bS^{d}) = \omega_d \left(1-\binom{d+1}{3}\pi^2 n^{-2} + O(n^{-3})\right),
$$
where $S(K)$ denotes the surface area of the spherical polytope $K$. On the other hand, the relation $2\omega_d h_d(C_n) = 2 \omega_d v_d(C_n) = S(C_n \cap \bS^{d})$ and Theorem~\ref{thm:ConicalIntrinsicVolumes} with $k=d-1$ yield
$$
\E S(C_n \cap \bS^{d}) = \omega_d \left(1 - 2 B_{2,d} n^{-2} + o(n^{-2})\right).
$$
Comparing both asymptotic relations, we obtain the required formula for $B_{2,d}$.
\end{proof}

Let us consider the special case $d=2$, where $B_{2,2} =\frac 12 \pi^2$ and hence
\begin{multline*}
\lim_{n\to\infty} \E f_0(C_n\cap \bS_+^2) = \lim_{n\to\infty} \E f_1(C_n\cap \bS_+^2)
=
\E f_0(\conv \Pi_{2,1}(c)) = \E f_1(\conv \Pi_{2,1}(c)) = \frac 12 \pi^2,
\end{multline*}
with $c>0$ being arbitrary.
For $d=3$, the identities $B_{3,3}=4\pi^2$ and $B_{2,3}=2\pi^2$ (following from~\eqref{eq:ConstantBd} and Proposition~\ref{prop:B_2_d}) combined with the Euler relation $f_0-f_1+f_2=2$ yield
\begin{multline*}
\lim_{n\to\infty} (\E f_0(C_n\cap \bS_+^3), \E f_1(C_n\cap \bS_+^3), \E f_2(C_n\cap \bS_+^3))
\\
=
(\E f_0(\conv \Pi_{3,1}(c)), \E f_1(\conv \Pi_{3,1}(c)), \E f_2(\conv \Pi_{3,1}(c)))
=
\left(2+\frac 23 \pi^2, 2\pi^2, \frac 4 3 \pi^2\right).
\end{multline*}

As a complement to Theorem \ref{thm:ConicalIntrinsicVolumes} we have the following weak limit theorem. This extends Theorem \ref{theo:conv_of_angle} to all Grassmann angles, which appears as the special case $k=d$.

\begin{theorem}\label{thm:WeakConvergenceGrassmannAngles}
For all  $k\in\{0,1,\ldots,d\}$ we have that
\begin{align*}
n^{d+1-k}\Big({1\over 2}- h_{k+1}(C_n)\Big) \overset{{\rm d}}{\longrightarrow} {1\over 2}\Big({2\over\omega_{d+1}}\Big)^{d+1-k} &\int_{(\R^d)^{d+1-k}}\ind_{\{\conv \Pi_{d,1}(2)\cap\aff\{x_1,\ldots,x_{d+1-k}\} = \varnothing\}}\\
&\qquad\qquad\qquad\times\prod_{i=1}^{d+1-k}{\dd x_i\over\|x_i\|^{d+1}}\,,
\end{align*}
as $n\to\infty$.
\end{theorem}

\subsection{Convex hull of Poisson point process with power-law intensity}\label{subsec:ConvexHullPPP}
We are now going to state explicit formulae for the expected values of some functionals of the random polytopes $\conv \Pi_{d,\gamma}(c)$ introduced in Section~\ref{subsec:weak_conv_thm}.

\subsubsection{Expectation of the \texorpdfstring{$T$}{T}-functional}
The results are most conveniently expressed via the \textit{$T$-functional} introduced by Wieacker~\cite{wieacker}. For a polytope $P\subset \R^d$, $k\in\{0,1,\ldots,d-1\}$ and for $a,b\geq 0$ it is defined by
$$
T_{a,b}^{d,k}(P) = \sum_{F \in \mathcal{F}_{k}(P)} \dist^a (F) \Vol_k^b(F),
$$
where $\dist(F)$ is the distance from the origin to the affine hull of the $k$-face $F$. The next theorem provides an explicit formula for the expected $T$-functional with $k=d-1$ of the random polytopes $\conv \Pi_{d,\gamma}(c)$.

\begin{theorem}\label{theo:T_expect}
For every $\gamma>0$, $c>0$ and all $a,b\geq 0$ such that $(\gamma-b)d + b - a>0$ and $\gamma-b>0$, we have that
\begin{multline*}
\E T_{a,b}^{d,d-1}(\conv\Pi_{d,\gamma}(c))
=
\frac{c^{d} \omega_d }{\gamma\,d!\,\omega_{\gamma+1}^d}
 \left(\frac{c}{\gamma \omega_{\gamma+1}}\right)^{\frac{a-b+ (b-\gamma)d}{\gamma}} \Gamma \left(\frac{(\gamma-b)d + b - a}{\gamma}\right)
 \\
\times\frac 1 {((d-1)!)^{b}}
\frac{\Gamma\left(\frac{\gamma-b}{2}d + \frac{b+1}{2} \right)}{\Gamma\left(\frac{\gamma-b}{2}d \right)} \left( \frac{\Gamma\left(\frac{\gamma-b}{2} \right)}{\Gamma\left(\frac{\gamma+1}{2} \right)} \right)^{d}\; \prod_{i=1}^{d-1} \frac{\Gamma\left( \frac{i+b+1}{2} \right)}{\Gamma\left( \frac{i}{2} \right)}.
\end{multline*}
If $(\gamma-b)d + b - a\leq 0$ or $\gamma-b\leq 0$, then the expectation equals $+\infty$.
\end{theorem}
Inserting special values for the parameters $a$ and $b$ leads to some interesting consequences.

\subsubsection{Expected number of faces}
Taking $a=b=0$, and observing that almost surely
$$
T_{0,0}^{d,d-1}(\conv\Pi_{d,\gamma}(c))=f_{d-1}(\conv\Pi_{d,\gamma}(c)),
$$
we obtain after simplification the following result for the mean number of facets of $\conv\Pi_{d,\gamma}(c)$.

\begin{corollary}\label{cor:hyperfaces_expect}
For every $\gamma>0$ and $c>0$, we have that
$$
\E f_{d-1}(\conv\Pi_{d,\gamma}(c))=\frac 2d \gamma^{d-1}\pi^{\frac{d-1}{2}} \frac{\Gamma(\frac{\gamma d+1}{2})}{\Gamma(\frac{\gamma d}{ 2})}\left(\frac{\Gamma(\frac{\gamma}{2})}{\Gamma(\frac{\gamma+1}{2})}\right)^d,
$$
independently of the parameter $c>0$.
\end{corollary}

\begin{remark}
All faces of the polytope $\conv\Pi_{d,\gamma}(c)$ are simplices with probability $1$. The Dehn--Sommerville relation
$$
d f_{d-1}(\conv\Pi_{d,\gamma}(c))= 2 f_{d-2}(\conv\Pi_{d,\gamma}(c))
$$
allows to compute the expected number of $(d-2)$-faces of $\conv\Pi_{d,\gamma}(c)$, but computing the expected number of $k$-faces for general $k$ remains an open problem.
\end{remark}

In particular, for $\gamma=1$ we obtain
$$
\E f_{d-1}(\conv\Pi_{d,1}(c)) = {2\pi^{d-{\frac 12}}\over d}{\Gamma({d+1\over 2})\over\Gamma({d\over 2})} = \pi^{d-{1\over 2}}{\Gamma({d+1\over 2})\over\Gamma(1+{d\over 2})}
$$
for all $c>0$. Using Legendre's duplication formula for the gamma function this can be rewritten as follows:
\begin{align*}
\pi^{d-{1\over 2}}{\Gamma({d+1\over 2})\over\Gamma(1+{d\over 2})} &= \pi^{d-{1\over 2}}{\Gamma({d+1\over 2})\over\Gamma(1+{d\over 2})}{\Gamma(1+{d\over 2})\over\Gamma(1+{d\over 2})} = {d\pi^{d-{1\over 2}}\over 2}{\Gamma({d+1\over 2})\Gamma({d\over 2})\over\Gamma(1+{d\over 2})^2} \\
&= {d\pi^{d-{1\over 2}}\over 2}{\Gamma(d)\sqrt{2\pi}\,2^{-d+{1\over 2}}\over\Gamma(1+{d\over 2})^2} = {2^{-d}\pi^d\,d!\over\Gamma(1+{d\over 2})^2} = 2^{-d}d!\kappa_d^2.
\end{align*}
This coincides with the limit in \eqref{eq:barany_fd_asympt} and is consistent with Theorem~\ref{theo:conv_of_faces}. More generally, for any $a\in[0,d)$ we have the explicit formula
$$
\E T_{a,0}^{d,d-1}(\conv\Pi_{d,1}(c)) = 2^{1-2a}c^a\,\Big({\pi\over 2}\Big)^{d-a}{\Gamma(d-a)\over\Gamma(1+{d\over 2})\Gamma({d\over 2})}.
$$

Another special case in which the formula from Corollary~\ref{cor:hyperfaces_expect} simplifies is $\gamma=2$. After simple transformations we obtain
$$
\E f_{d-1}(\conv\Pi_{d,2}(c)) = \binom {2d}{d}.
$$
In dimension $d=2$ this means that the expected number of edges (or vertices) of the convex hull of the Poisson point process with intensity $\|x\|^{-4}$ in $\R^2$ is $6$, a fact due to Rogers~\cite{rogers}. For $d=3$ we obtain that the expected number of faces of the convex hull of the Poisson point process with intensity $\|x\|^{-5}$ is $20$. Since the faces are simplices a.s., the relation $3f_2 = 2f_1$ holds, which together with the Euler relation $f_0-f_1+f_2 = 2$ yields that the expected number of edges (respectively, vertices) is $30$ (respectively, $12$). To summarize, the expected $f$-vector of $\conv \Pi_{3,2}$ is the same as the $f$-vector of the regular icosahedron.

Finally, observe that in the case $d=2$ and for arbitrary $\gamma>0$, Corollary~\ref{cor:hyperfaces_expect} can be written as
$$
\E f_{1}(\conv\Pi_{2,\gamma}(c)) = \E f_{0}(\conv\Pi_{2,\gamma}(c)) = 4\pi \frac{B\left(\frac 12, \gamma + \frac 12 \right)}{B^2 \left(\frac 12, \frac{\gamma+1}{2}\right)},
$$
where $B$ denotes the Beta function. This formula is due to~\citet[Theorem~4.4]{Davis+Mulrow+Resnick:1987}; see also~\citet{carnal} where a similar formula is derived for convex hulls of i.i.d.\ samples with spherically symmetric regularly varying distributions.

\subsubsection{Expected volume}
Let us compute the expected volume of $\conv \Pi_{d,\gamma}(c)$. Since the origin is a.s.\ in the interior of $\conv\Pi_{d,\gamma}(c)$, we have that
$$
\Vol_d (\conv\Pi_{d,\gamma}(c)) = {1\over d}\,T_{1,1}^{d,d-1}(\conv\Pi_{d,\gamma}(c)),
$$
which together with Theorem~\ref{theo:T_expect} leads to the following result for the mean volume of the convex hull of $\Pi_{d,\gamma}(c)$.

\begin{corollary}\label{cor:volume_expect}
For every $\gamma>1$ and $c>0$ we have that
$$
\E \Vol_d (\conv\Pi_{d,\gamma}(c))  = {c^{d\over \gamma}\over d! 2^{d(1+{1\over \gamma})}\pi^{d\over 2\gamma}}\left({\gamma\over\Gamma({\gamma+1\over 2})}\right)^{d(\gamma-1)\over{\gamma}}{\Gamma(1+d-{d\over\gamma})\Gamma({\gamma-1\over 2})^d\over{\Gamma(1+{d\over 2})}}.
$$
For $0<\gamma \leq 1$ we have $\E \Vol_d (\conv\Pi_{d,\gamma}(c))=+\infty$.
\end{corollary}
We remark that in the special case $\gamma=2$ the formula become particularly simple:
$$
\E \Vol_d (\conv\Pi_{d,2}(c)) = \frac{1}{d!}\left(\frac c 2 \right)^{d/2}.
$$

\subsubsection{Expected intrinsic volumes}
 We compute the expected values of the intrinsic volumes $V_k(\conv\Pi_{d,\gamma}(c))$, $k\in\{0,\ldots,d\}$,  of the random polytopes $\conv\Pi_{d,\gamma}(c)$. We recall from~\cite{SchneiderBook} or~\cite[Eqn.~(6.11) on page 222]{SW08} that the \textit{intrinsic volume} of degree $k\in\{0,1,\ldots,d\}$ of a compact convex set $K\subset\R^d$ is given by
$$
V_k(K) := {d\choose k}{\kappa_d\over\kappa_k\kappa_{d-k}}\int_{G(d,k)}\Vol_k(P_L K)\,\nu_k(\dd L),
$$
where, as above, $G(d,k)$ is the Grassmannian of all $k$-dimensional linear subspaces of $\R^d$ with the unique Haar probability measure $\nu_k$ and $P_L K$ is the orthogonal projection of $K$ onto $L$. For example $V_0(K)=\ind_{\{K\neq\varnothing\}}$, $V_1(K)$ is a constant multiple of the mean width, $2V_{d-1}(K)$ is surface area and $V_d(K)$ is just the volume of $K$.

\begin{proposition}\label{cor:IntrinsicVolumes}
For every $\gamma>1$, $c>0$ and $k\in\{1,\ldots,d\}$ we have that
$$
\E V_k (\conv\Pi_{d,\gamma}(c))  = {d\choose k}{\kappa_d\over\kappa_k\kappa_{d-k}}{c^{k\over \gamma}\over 2^{k(1+{1\over \gamma})}\pi^{k\over 2\gamma}}\left({\gamma\over\Gamma({\gamma+1\over 2})}\right)^{k(\gamma-1)\over{\gamma}}{\Gamma(1+k-{k\over\gamma})\Gamma({\gamma-1\over 2})^k\over{\Gamma(1+{k\over 2})}}.
$$
For $0<\gamma \leq 1$ we have $\E V_k (\conv\Pi_{d,1}(c)) = +\infty $ for all $k\in\{1,\ldots,d\}$.
\end{proposition}

\subsubsection{Symmetric convex hulls}
The symmetric convex hull $\sconv \Pi$ of a point process $\Pi$ is defined as the convex hull of the points of the form $\pm x$, where $x$ is a point of $\Pi$.  The next theorem evaluates the expected $T$-functional of $\sconv \Pi_{d,\gamma}(c)$.
\begin{theorem}\label{theo:T_expect_symm}
For every $\gamma>0$, $c>0$ and all $a,b\geq 0$ such that $(\gamma-b)d + b - a>0$ and $\gamma-b>0$, we have that
$$
\E T_{a,b}^{d,d-1}(\sconv\Pi_{d,\gamma}(c))
=
\E T_{a,b}^{d,d-1}(\conv\Pi_{d,\gamma}(2c)).
$$
\end{theorem}
It is now straightforward to state the formulae for the expected facet number, volume, and intrinsic volumes of the symmetric convex hull of $\Pi_{d,\gamma}(c)$.

\section{Background material from stochastic geometry and theory of random measures}\label{sec:BackgroundMaterial}

In order to keep the paper self-contained we collect in this section some background material that is used throughout the proofs.

\subsection{Convergence of measures}
Let $S$ be a locally compact metric space. We denote by $\cM_S$ (respectively, $\cN_S$) the space of locally finite (respectively, locally finite integer-valued) measures on $S$. We supply $\cM_S$ and $\cN_S$ with the topology of vague convergence and recall that a sequence $(\mu_n)_{n\in\N}\subset\cM_S$ vaguely converges to a measure $\mu\in\cM_S$ provided that
$$
\lim_{n\to\infty}\int_S f(x)\,\mu_n(\dd x) = \int_S f(x)\,\mu(\dd x)
$$
for all continuous functions $f:S\to[0,\infty)$ with compact support. We shall write $\mu_n\overset{\rm v}{\longrightarrow}\mu$ in such a case. It is known from~\cite[Lemma 15.7.4]{Kallenberg:1983} that $\cN_S$ is a vaguely closed subset of $\cM_S$.

The vague topology turns $\cM_S$ and $\cN_S$ into Polish spaces (see~\cite[Lemma 15.7.7]{Kallenberg:1983}).
A \textit{random measure} (respectively, a \textit{point process}) is a random variable, defined on some probability space $(\Omega,\cA,\P)$, and taking values in $\cM_S$ (respectively, $\cN_S$).  In this paper we denote by  $\eta_n\overset{\rm w}{\longrightarrow}\eta$ the weak convergence of a sequence $(\eta_n)_{n\in\N}$ of random measures on $S$  to another random measure $\eta$, as $n\to\infty$.

\subsection{Poisson point processes}
Let $\mu$ be a locally finite measure on $S$ without atoms. A \textit{Poisson point process} $\Pi$ on $S$ with intensity measure $\mu$ is a random variable defined on some probability space $(\Omega,\cA,\P)$ taking values in the measurable space $\cN_S$ such that (i) $\Pi(B)$ is a Poisson random variable with mean $\mu(B)$ for each Borel set $B\subset S$ and (ii) the random variables $\Pi(B_1), \ldots, \Pi(B_n)$ are independent whenever the Borel sets $B_1,\ldots,B_n\subset S$ are pairwise disjoint. We remark that almost surely $\Pi$ can be represented as $\Pi=\sum_{i=1}^\kappa\delta_{x_i}$ with random points $x_1,x_2,\ldots\in S$ and a Poisson random variable $\kappa$ with mean $\mu(S)$ (which is interpreted as $+\infty$ if $\mu$ is not a finite measure). Here, $\delta_x$ stands for the unit mass at $x\in S$.

Let $k\in\N$ and denote by $\Pi_{\neq}^k$ the collection of $k$-tuples of distinct points charged by $\Pi$. It is a crucial fact that the Poisson point process $\Pi$ satisfies the multivariate \textit{Mecke equation}
\begin{multline}\label{eq:MeckeEquation}
\E \sum_{(x_1,\ldots,x_k)\in\Pi_{\neq}^k} f(x_1,\ldots,x_k;\Pi)\\
 = \int_{S}\cdots\int_{S} \E f(x_1,\ldots,x_k;\Pi+\delta_{x_1}+\ldots+\delta_{x_k})\,\mu(\dd x_1)\ldots\mu(\dd x_k)
\end{multline}
for any non-negative measurable function $f:S^k\times\cN_S\to\R$; see~\cite[Corollary 3.2.3]{SW08}. Here, $\E$ denotes expectation (i.e.\ integration) with respect to $\P$.

\subsection{Polytopes and cones}
For a set $A\subset\R^d$, we denote by $\conv A$ the convex hull of $A$. In particular, if $A$ is a finite set, $\conv A$ is called a (convex) \textit{polytope}. A \textit{face} of a polytope (or a general closed convex set) $P\subset\R^d$ is the intersection of $P$ with one of its supporting hyperplanes (which are hyperplanes $H$ intersecting the boundary of $P$ and having the property that $P$ is entirely contained in one of the closed half-spaces bounded by $H$). If the affine hull of a face has dimension $k\in\{0,1,\ldots,d-1\}$ we call it a \textit{$k$-face} or a face of dimension $k$. By $\cF_k(P)$ we denote the set of $k$-faces of a polytope $P$ and by $f_k(P)$ its cardinality. The \textit{$f$-vector} ${\bf f}(P)$ of $P$ is given by ${\bf f}(P):=(f_0(P),\ldots,f_{d-1}(P))$.

A (polyhedral) \textit{cone} is an intersection of finitely many closed half-spaces whose boundaries pass through the origin. If $C\subset\R^d$ is a polyhedral cone, we denote by $f_k(C)$ the number of $k$-dimensional faces of $C$, for $k\in \{0,1,\ldots,d-1\}$.

\subsection{Weak convergence of random compact convex sets}
We denote by $\cC^d$ the space of compact subsets of $\R^d$. The \textit{Hausdorff distance} ${\rm d}_H$ on this space is defined as
$$
{\rm d}_H(C,C') := \max\big\{\max_{x\in C}\min_{y\in C'}\|x-y\|,\max_{x\in C'}\min_{y\in C}\|x-y\|\big\},\qquad C,C'\in\cC^d,
$$
where $\|\,\cdot\,\|$ is the Euclidean norm on $\R^d$. We shall use the notation $C_n\overset{{\rm d}_H}{\longrightarrow}C_0$ to indicate that ${\rm d}_H(C_n,C_0)\to 0$, as $n\to\infty$, for a sequence $(C_n)_{n\in\N_0}\subset\cC^d$.

By $\cK^d$ we denote the space of compact convex subsets of $\R^d$, which is a closed subspace of $\cC^d$ with respect to the Hausdorff distance.
A \textit{random compact convex set} is a random variable $X$, defined on some probability space $(\Omega,\cA,\P)$, which takes values in the measurable space $\cK^d$. It is known from~\cite[Theorem 7.8]{molchanov_book} that the distribution of such a random set is uniquely determined by its containment functional
$$
C_X(K):=\P(X\subseteq K),\qquad K\in\cK^d.
$$
Distributional convergence of a sequence $(X_n)_{n\in\N}$ of random compact convex sets to another random compact convex set $X_0$ can be formulated in terms of the convergence of the containment functionals as follows. Namely, $X_n$ converges in distribution to $X_0$ weakly on $\cK^d$, as $n\to\infty$, if and only if $\lim_{n\to\infty} C_{X_n}(K) = C_{X_0}(K)$ for all $K\in\cK^d$ for which $C_{X_0}(K)=C_{X_0}({\rm int}(K))$, where ${\rm int}(K)$ denotes the interior of $K$; see~\cite[Theorem 7.12]{molchanov_book}. We shall indicate such convergence by $X_n\overset{\rm w}{\longrightarrow}X_0$ in this paper.

The next standard lemma will be used in the proof 
of Theorem~\ref{theo:conv_of_angle}. 
\begin{lemma}\label{lem:hausdorff_indicator}
Let $(K_n)_{n\in\N_0}\subset \mathcal K^d$ be deterministic compact convex sets such that $K_n \overset{{\rm d}_H}{\longrightarrow} K_0$.   Then, for every $x\in \R^d\backslash\partial K_0$ we have $\lim_{n\to\infty} \ind_{K_n}(x) = \ind_{K_0}(x)$.
\end{lemma}
\begin{proof}
Assume first $x\notin K_0$. Then there is a hyperplane $H$ such that $x$ and $K_0$ are contained in different open half-spaces $H_{>0}$ and $H_{<0}$ defined by $H$. For sufficiently small $\eps>0$, the $\eps$-neighborhood of $K_0$ is still contained in $H_{<0}$. Hence, for sufficiently large $n$, we have $K_n \subset H_{<0}$ and at the same time $x\in H_{>0}$. It follows that $\ind_{K_n}(x) = 0 = \ind_{K_0}(x)$ for sufficiently large $n$, which proves the claim. Suppose now that $x$ is in the interior of $K_0$ and without loss of generality that $x=0$.  We argue by contradiction and assume that $0\notin K_n$ for infinitely many $n$. By the hyperplane separation theorem, there is a unit vector $\theta_n\in\R^d$ such that $\langle z, \theta_n\rangle <0$ for all $z\in K_n$. By passing to a subsequence we may assume that $\theta_n\to \theta$ for some unit vector $\theta\in \R^d$. Since $0$ is in the interior of $K_0$, we can find $\eps>0$ such that $\eps \theta \in K_0$. The distance between $\eps \theta$ and $K_n$ is bounded from below by the distance between $\eps \theta$ and the half-space $\{z\in \R^d\colon \langle z, \theta_n \rangle <0\}$ containing $K_n$. Thus, the distance between $\eps \theta$ and $K_n$ is at least $\langle \eps \theta, \theta_n \rangle$ which is  larger than  $\eps /2$ for sufficiently large $n$. Therefore, $\eps \theta \in K_0$ but at the same time $\eps \theta$ is not contained in the $\eps/2$-neighborhood of $K_n$, a contradiction to the assumption $K_n \to K_0$ in the Hausdorff distance.
\end{proof}

When dealing with weak convergence, we shall frequently make use of the following Skorokhod representation theorem~\cite[Theorem~4.30 on p.~79]{KallenbergFMP}: Assume that $(X_n)_{n\in\N_0}$ are random elements with values in a separable metric space and $X_n$ converges weakly to $X_0$ as $n\to\infty$. Then there exist random elements $(X_n')_{n\in\N_0}$ defined on a common probability space $(\Omega',\mathcal{F}',\mathbb{P}')$ such that $X_n'$ has the same distribution as $X_n$ for all $n\in\N_0$, and $X_n'$ converges to $X_0'$ $\mathbb{P}'$-a.s.

\subsection{The affine Blaschke--Petkantschin formula}

For $k\in\{0,1,\ldots,d\}$ we let $G(d,k)$ and $A(d,k)$ be the spaces of $k$-dimensional linear and affine subspaces of $\R^d$, respectively. By $\nu_k$ we denote the unique probability measure on $G(d,k)$ which is invariant under the action of ${\rm SO}(d)$. The invariant measure $\mu_k$ on $A(d,k)$ is then given by
\begin{equation}\label{eq:DefMeasureMuk}
\mu_k(\,\cdot\,) = \int_{G(d,k)}\int_{L^\bot}{\bf 1}\{L+x\in\,\cdot\,\}\,\lambda_{L^\bot}(\dd x) \nu_{k}(\dd L),
\end{equation}
where $\lambda_{L^\bot}$ denotes the Lebesgue measure on $L^\bot$; see~\cite[pp.~168--169]{SW08}. Similarly, we shall write $\lambda_E$ for the Lebesgue measure on $E\in A(d,k)$.

The \textit{affine Blaschke--Petkantschin} formula is a so-called integral-geometric transformation formula and reads as follows; see~\cite[Theorem 7.2.7]{SW08}. For any non-negative measurable function $f:(\R^d)^{k+1}\to\R$ one has that
\begin{align}\label{eq:BlaschkePetkantschin}
&\int_{(\R^d)^{k+1}}f(x_0,\ldots,x_k)\,\dd(x_0,\ldots,x_k)\\
\nonumber & = b_{d,k}(k!)^{d-k}\int_{A(d,k)}\int_{E^{k+1}}f(x_0,\ldots,x_k)\,\Delta_k(x_0,\ldots,x_k)^{d-k}\,\lambda_E^{k+1}(\dd(x_0,\ldots,x_k))\mu_k(\dd E),
\end{align}
where the constant $b_{d,k}$ is given by
$$
b_{d,k} = \frac{\omega_{d+1-k}\cdots\omega_d}{\omega_1\cdots\omega_k}.
$$

\section{Proofs: Weak limit theorems and convergence of moments}\label{sec:ProofsWeakLimitTheorems}
\subsection{Continuity of functionals}
Our next lemma is an essential ingredient in the proof of Theorem~\ref{theo:conv_of_polytopes}, Theorem~\ref{theo:conv_of_faces} and Theorem~\ref{theo:conv_of_angle}. Let us recall that we denote by
$\cN:=\cN_{\R^d\cup\{\infty\}\backslash\{0\}}$ the space of locally finite integer-valued measures on $\R^d\cup\{\infty\}\backslash\{0\}$, where $\R^d\cup\{\infty\}$ is a one-point compactification of $\R^d$.

\begin{lemma}\label{lem:basic_convergence}
Assume that $(\eta_n)_{n\in\mathbb{N}_0}$ is a sequence of deterministic measures in $\cN$ and suppose that $\eta_n \overset{\rm v}{\longrightarrow}  \eta_0$, as $n\to\infty$. Suppose further that $\eta_0$ satisfies $\eta_0(\{\infty\}) =0$ and that the following two conditions are satisfied:
\begin{itemize}
\item[(a)] $\eta_0(H_+)>0$ for every open half-space $H_{+}\subset \mathbb{R}^d$ such that $0\in \partial H_+$,
\item[(b)] the atoms of $\eta_0$ are in general position, that is, no $k+2$ atoms of $\eta_0$ lie in the same $k$-dimensional affine subspace for all $k=1,\ldots,d-1$.
\end{itemize}
Then, $\conv \eta_0$ is a convex polytope containing $0$ in its interior. Moreover, as $n\to\infty$, we have the convergence
$$
\conv\eta_n \overset{{\rm d}_H}{\longrightarrow}  \conv\eta_0
$$
on the space $\cK^d$ as well as the convergence of the $f$-vectors
$$
{\bf f}(\conv\eta_n) \longrightarrow {\bf f}(\conv\eta_0).
$$
\end{lemma}
\begin{proof}
Since the ball $B_r(0)$ is open, the set $\R^d\cup\{\infty\}\backslash B_r(0)$ is compact for every $r>0$ and thus $\eta_0$ has only a finite number of atoms outside of $B_r(0)$ by the local finiteness of $\eta_0$. Since $\eta_0(\{\infty\})=0$, the set of atoms of $\eta_0$ is bounded. Hence, $\conv \eta_0$ is a compact convex set. We show that it is in fact a polytope. By the supporting hyperplane theorem (see \cite[Chapter 1.3]{SchneiderBook}), Assumption (a) implies that the origin $0$ is an interior point of $\conv\eta_0$. Thus, there exists an open ball $B_{2r}(0)\subset \conv\eta_0$ with $r>0$. Let us denote by $A_1,\ldots,A_k$ the (finitely many) atoms of $\eta_0$ outside of $B_r(0)$. We claim that
\begin{equation}\label{eq:convex_hull_eta0}
\conv\eta_0=\conv\{A_1,\ldots,A_k\}
\end{equation}
and, in particular, $\conv\eta_0$ is a convex polytope.
To prove~\eqref{eq:convex_hull_eta0}, it suffices to show that $B_r(0)\subset \conv \{A_1,\ldots,A_k\}$. Assume that $x\in B_r(0)$ but $x\notin \conv \{A_1,\ldots,A_k\}$. By the separating hyperplane theorem (see again \cite[Chapter 1.3]{SchneiderBook}), there is an open half-space $G_+$ such that $x\notin G_+$ and  $\conv\{A_1,\ldots,A_k\} \subset G_+$. After applying an orthogonal transformation, we may assume that $G_+ = \{y\in\R^d \colon y_1 < a\}$, where $y_1$ is the first coordinate of $y\in\R^d$. Since $x\notin G_+$, its first coordinate satisfies $x_1\geq a$, hence $a < r$. Now,
$$
\conv \eta _0 \subset \conv (\{A_1,\ldots,A_k\} \cup B_r(0)) \subset \conv (G_+ \cup B_r(0)) \subset \{y\in \R^d\colon y_1\leq r\},
$$
which is in contradiction with $B_{2r}(0) \subset \conv \eta_0$. This proves~\eqref{eq:convex_hull_eta0}.

By Proposition 3.13 in~\cite{sR87}, the assumed vague convergence of $\eta_n$ to $\eta_0$, as $n\to\infty$, implies that for sufficiently large $n$, each $\eta_n$ has exactly $k$ atoms, say $\{A_1^{(n)},\ldots,A_k^{(n)}\}$, in $\R^d\backslash \overline{B_r(0)}$ and
\begin{equation}\label{eq:conv_of_extreme_points}
\{A_1^{(n)},\ldots,A_k^{(n)}\} \overset{{\rm d}_H}{\longrightarrow} \{A_1,\ldots,A_k\},
\end{equation}
as $n\to\infty$, on the space $\cC^d$. Since the mapping $\conv:\cC^d\to\cC^d$ is continuous with respect to the Hausdorff distance (see~\cite[Theorem 12.3.5]{SW08}), we also have that
$$
\conv\{A_1^{(n)},\ldots,A_k^{(n)}\} \overset{{\rm d}_H}{\longrightarrow} \conv \{A_1,\ldots,A_k\},
$$
as $n\to\infty$, on the space $\cC^d$ as well as on the space $\cK^d$. Now, since $B_{2r}(0)\subset\conv\eta_0=\conv\{A_1,\ldots,A_k\}$, this yields that $B_r(0)\subset \conv\{A_1^{(n)},\ldots,A_k^{(n)}\}$ for large $n$ and therefore,
\begin{equation}\label{eq:convex_hull_etan}
\conv\eta_n = \conv\{A_1^{(n)},\ldots,A_k^{(n)}\},
\end{equation}
for all sufficiently large $n$, which can be proved in the same way as~\eqref{eq:convex_hull_eta0}.

Assumption (b) implies that the points of $\{A_1,\ldots,A_k\}$ are in general position, which in conjunction with \eqref{eq:conv_of_extreme_points} yields that also the points of $\{A_1^{(n)},\ldots,A_k^{(n)}\}$ are in general position for sufficiently large $n$. Therefore, \eqref{eq:conv_of_extreme_points} implies that for each $k\in\{0,1,\ldots,d-1\}$ the number of $k$-dimensional faces of $\conv\{A_1^{(n)},\ldots,A_k^{(n)}\}$ is the same as the number of $k$-dimensional faces of
$\conv\{A_1,\ldots,A_k\}$ for all $k\in\{0,\ldots,d-1\}$ and large enough $n$. This completes the proof of the lemma.
\end{proof}

Since for each $\gamma>0$ and $c>0$, the Poisson point process $\Pi_{d,\gamma}(c)$ is an element of the space $\cN$ and almost surely satisfies the assumptions (a) and (b), Lemma~\ref{lem:basic_convergence} yields the following result.

\begin{corollary}\label{cor:polytope}
For each $\gamma>0$ and $c>0$, $\conv\Pi_{d,\gamma}(c)$ is almost surely a convex polytope containing $0$ in its interior.
\end{corollary}

\subsection{Proofs of weak limit theorems}

We are now ready to prove Theorems~\ref{theo:conv_of_polytopes}, \ref{theo:conv_of_faces} and  \ref{theo:conv_of_angle}. The proofs rely on some auxiliary lemmas that are postponed to Section \ref{sec:AuxLemmas} at the end of the paper.

\begin{proof}[Proof of Theorem~\ref{theo:conv_of_polytopes}]
Recall that the mapping $\mathcal{P}:\mathbb{S}^{d}_{+}\cap \{x_0>0\}\to \mathbb{R}^{d}$ was defined by the equality \eqref{eq:DefMappingP}. For each $i\in\{1,\ldots,n\}$ let $\ell_i$ be the line in $\mathbb{R}^{d+1}$ passing through the origin and the point $U_i$. This line intersects the hyperplane $H_1:=\{x_0=1\}$ at the point $(1,\mathcal{P}(U_i))\in H_1$. This observation implies that
$$
C_n\cap H_1 = \conv \{(1,\mathcal{P}(U_i))\colon i=1,\ldots,n\}
$$
and, therefore,
\begin{equation}\label{eq:representation}
(T_nC_n\cap H_1) - e_0 = \conv\{n^{-1}\mathcal{P}(U_i)\colon i=1,\ldots,n\}.
\end{equation}
Hence, it is enough to show that
\begin{equation}\label{eq:resnick_conv_of_ch}
\conv \{n^{-1}\mathcal{P}(U_i)\colon i=1,\ldots,n\} \overset{\rm w}{\longrightarrow} \conv\Pi_{d,1}(2)
\end{equation}
on the space $\cK^d$. To prove this, we first note that as a consequence of Lemma~\ref{lem:lemma_reg_var} below and Proposition 3.21 in~\cite{sR87} we have
\begin{equation}\label{eq:resnick_conv_of_pp}
\sum_{i=1}^{n}\delta_{n^{-1}\mathcal{P}(U_i)} \overset{\rm w}{\longrightarrow} \Pi_{d,1}(2),\qquad\text{as }n\to\infty,
\end{equation}
weakly on the space $\mathcal N_{\R^d\backslash\{0\}}$. Now we can use the Skorokhod representation theorem to pass to the a.s.\ convergence on a new probability space, and then apply Lemma~\ref{lem:basic_convergence} pointwise. Going back to the original probability space, we  get the required convergence~\eqref{eq:resnick_conv_of_ch}.  The proof of Theorem~\ref{theo:conv_of_polytopes} is thus complete. 
\end{proof}

\begin{remark}
For $d=2$ the convergence \eqref{eq:resnick_conv_of_ch} also follows from Theorem 3.1 in~\cite{Davis+Mulrow+Resnick:1987}.
\end{remark}

\begin{proof}[Proof of Theorem~\ref{theo:conv_of_faces}]
Let $k\in\{1,\ldots,d\}$. From \eqref{eq:representation} we obtain the almost sure equality
$$
f_{k-1}(C_n \cap \mathbb S_+^d) = f_k(C_n)=f_{k-1}((T_nC_n\cap H_1)-e_0)=f_{k-1}( \conv\{n^{-1}\mathcal{P}(U_i)\colon i=1,\ldots,n\}).
$$
Passing in~\eqref{eq:resnick_conv_of_pp} to the a.s.\ convergence by the Skorokhod representation theorem, using Lemma~\ref{lem:basic_convergence} pointwise, and returning back to the original probability space yields 
$$
\mathbf f(\conv\{n^{-1}\mathcal{P}(U_i)\colon i=1,\ldots,n\})
 \overset{\rm d}{\longrightarrow}
\mathbf f(\conv\Pi_{d,1}(2)),
$$
which proves the desired statement. 
\end{proof}

\begin{proof}[Proof of Theorem~\ref{theo:conv_of_angle}]
We shall use the following alternative definition of the solid angle. For a convex cone $C\subset \{x_0\geq 0\} \subset \R^{d+1}$ the solid angle equals
$$
\alpha(C)={1\over 2}\,\mathbb{P}(U\in C\cap \mathbb{S}^d_{+}),
$$
where $U$ is a random vector with the uniform distribution on the half-sphere $\mathbb{S}^d_{+}$. We have
$$
2n\left(\frac{1}{2}-\alpha(C_n)\right)
=
n\big(1-\P(U\in C_n\cap\mathbb{S}_+^d | C_n) \big) = n\mathbb{P}(U\notin C_n\cap \mathbb{S}^d_{+}|C_n),
$$
where $U$ is independent of $C_n$ and $\P(\,\cdot\, |\,\cdot\,)$ denotes conditional probability.  Further,
\begin{align*}
n\mathbb{P}\big(U\notin C_n\cap \mathbb{S}^d_{+} | C_n\big) &=n\mathbb{P}\big((1,\mathcal{P}(U))\notin C_n\cap H_1| C_n\big)\\
&=n\mathbb{P}\big(\mathcal{P}(U)\notin \conv \{\mathcal{P}(U_i)\colon i=1,\ldots,n\}|U_1,\ldots,U_n\big)\\
&=\mu_n(\mathbb{R}^d\backslash \conv \{n^{-1}\mathcal{P}(U_i)\colon i=1,\ldots,n\})
,
\end{align*}
where the measure $\mu_n$ is given by $\mu_n(\cdot):=n\mathbb{P}(n^{-1}\mathcal{P}(U)\in \,\cdot\,)$. As a consequence of Proposition~\ref{prop:cauchy}, the Lebesgue density of $\mu_n$ is given by
\begin{equation}\label{eq:Densitypn}
p_n(x) = \frac{2n^{d+1}}{\omega_{d+1} (1+ n^2 \|x\|^2)^{\frac{d+1}{2}}}. 
\end{equation}
Denoting the random polytope $\conv \{n^{-1}\mathcal{P}(U_i)\colon i=1,\ldots,n\}$ by $L_n$,  we can write
$$
2n\left(\frac{1}{2}-\alpha(C_n)\right)
=\int_{\R^d}  (1- \ind_{L_n}(x)) p_n(x) \dd x. 
$$
Let also $L_0 := \conv \Pi_{d,1}(2)$. From~\eqref{eq:resnick_conv_of_ch} we know that $L_n$ converges to $L_0$ weakly on the space $\mathcal K^d$. By the Skorokhod representation theorem, on a new probability space we can define random convex sets $(L_n')_{n\in\N_0}$ such that $L_n'$ has the same distribution as $L_n$, for all $n\in\N_0$, and with probability one $L_n' \to L_0'$ in the Hausdorff metric. Let us fix some outcome $\omega$ in the new probability space outside the event where the convergence fails to hold or where $L_0'$ is not a polytope containing $0$ in its interior. The probability of this exceptional event is $0$; see Corollary~\ref{cor:polytope}.  With this convention, the deterministic polytopes $L_n'(\omega)$ converge to $L_0'(\omega)$ in the Hausdorff metric. From Lemma~\ref{lem:hausdorff_indicator} it follows that
$$ 
\lim_{n\to\infty} \ind_{L_n'(\omega)}(x) =  \ind_{L_0'(\omega)}(x)
\text{ for all } x\in \R^d \backslash \partial L_0'(\omega).
$$
Note that the Lebesgue measure of $\partial L_0'(\omega)$ is $0$ because $L_0'(\omega)$ is a polytope. The density $p_n(x)$ satisfies
$$
\lim_{n\to\infty} p_n(x) = \frac2{\omega_{d+1}} \frac{1}{\|x\|^{d+1}}
\quad
\text{and}
\quad
p_n(x) \leq \frac2{\omega_{d+1}} \frac{1}{\|x\|^{d+1}}
$$
for all  $x\in \R^d\backslash\{0\}$.
Taking everything together, we obtain  that for Lebesgue-a.e.\ $x\in \R^d$,
$$
\lim_{n\to\infty} (1- \ind_{L_n'(\omega)}(x)) p_n(x) = (1-\ind_{L_0'(\omega)}(x)) \frac2{\omega_{d+1}} \frac{1}{\|x\|^{d+1}}. 
$$
Also, for sufficiently large $n$, we have the integrable bound
$$
(1- \ind_{L_n'(\omega)}(x)) p_n(x) \leq  \ind_{\{\|x\|\geq r(\omega)/2\}} \frac2{\omega_{d+1}} \frac{1}{\|x\|^{d+1}},
$$
where $r(\omega)>0$ is the distance from $0$ to the boundary of $L_0'(\omega)$. The dominated convergence theorem yields
$$
\int_{\R^d}  (1- \ind_{L_n'(\omega)}(x)) p_n(x) \dd x 
\to 
\int_{\R^d}  (1-\ind_{L_0'(\omega)}(x)) \frac2{\omega_{d+1}} \frac{1}{\|x\|^{d+1}} \dd x 
=
\frac2{\omega_{d+1}} \int_{\R^d \backslash L_0'(\omega)}  \frac{\dd x }{\|x\|^{d+1}},
$$
as $n\to\infty$. 
We recall that this convergence holds for every outcome $\omega$ outside some event with probability zero. In particular, it implies the distributional convergence of the corresponding random variables. Returning back to the original probability space, we can replace $L_n'$ by $L_n$ for all $n\in\N_0$, thus obtaining
$$
2n\left(\frac{1}{2}-\alpha(C_n)\right) = \int_{\R^d}  (1- \ind_{L_n}(x)) p_n(x) \dd x 
 \overset{\rm d}{\longrightarrow}
\frac2{\omega_{d+1}} \int_{\R^d \backslash L_0}  \frac{\dd x }{\|x\|^{d+1}},
$$
as $n\to\infty$. Recall finally that $L_0 = \conv \Pi_{d,1}(2)$. The proof is complete. 
\end{proof}

\subsection{Convergence of moments: Proof of Theorem \ref{theo:ConvergenceOfExpectations}}
In view of Theorem \ref{theo:conv_of_faces} we need to show that the sequence $(f_{k}^m (C_n))_{n\in\N}$ is uniformly integrable for every $k=1,\ldots,d$ and $m\in\N$. This is equivalent to
\begin{equation}\label{eq:moment_boundedndess}
\sup_{n\in\N} \E f_{k}^m (C_n)<\infty
\end{equation}
for every $k=1,\ldots,d$ and $m\in\N$, because \eqref{eq:moment_boundedndess} for a fixed $m$ implies uniform integrability of $(f_{k}^\ell (C_n))_{n\in\N}$ for $0\leq\ell<m$.

To prove \eqref{eq:moment_boundedndess} we note that for an arbitrary (spherical) polytope  $P_n$ the number $f_k(P_n)$ of its $k$-dimensional faces satisfies
$$
f_k(P_n)\leq {f_0(P_n)\choose k+1}\leq f^{k+1}_0(P_n),\qquad k=0,\ldots,d-1.
$$
From this observation it follows that \eqref{eq:moment_boundedndess} is equivalent to
\begin{equation}\label{eq:moment_0_boundedndess}
\sup_{n\in\N} \E f_{0}^m (C_n\cap \bS_+^d)<\infty
\end{equation}
for every $m\in\N$. Recall that $\mathcal{P}:\mathbb{S}^{d}_{+}\cap \{x_0>0\}\to \mathbb{R}^{d}$ is the map defined by~\eqref{eq:DefMappingP}. Clearly, $f_{0} (C_n\cap \bS_+^d)$ coincides with the number of vertices of the convex hull of $\mathcal{P}(U_1),\ldots, \mathcal{P}(U_n)$ in $\R^d$. Write
\begin{align*}
\E f_{0}^m (C_n\cap \bS_+^d)&=
\E \left(\sum_{i=1}^{n}\ind_{\{\mathcal{P}(U_i)\notin \conv\{\mathcal{P}(U_j),j\neq i,j=1,\ldots,n\}\}}\right)^m\\
&=\sum_{i_1=1}^{n}\cdots\sum_{i_m=1}^{n}\P\left(\mathcal{P}(U_{i_k})\notin \conv\{\mathcal{P}(U_j),j\neq i_k,j=1,\ldots,n\},k=1,\ldots,m\right)\\
&\leq \sum_{i_1=1}^{n}\cdots\sum_{i_m=1}^{n}\P\left(\mathcal{P}(U_{i_1}),\mathcal{P}(U_{i_2}),\ldots,\mathcal{P}(U_{i_m})\notin \conv\{\mathcal{P}(U_j),j \notin\{i_1,i_2,\ldots,i_m\}\}\right).
\end{align*}
In view of this representation, the inequality \eqref{eq:moment_0_boundedndess} follows once we can show
that
$$
\P\left(\mathcal{P}(U_{1}),\mathcal{P}(U_{2}),\ldots,\mathcal{P}(U_{k})\notin \conv\{\mathcal{P}(U_{k+1}),\ldots,\mathcal{P}(U_n)\}\right)=O(n^{-k}),
$$
as $n\to\infty$, for every fixed $k\in\N$, where the constant in the Landau term $O(\,\cdot\,)$ might depend on $k$.  Denote by $K_n\subset\R^d$ the convex hull of the random points $\cP(U_1),\ldots,\cP(U_n)$. Fix $k\in\N$ and let $Y_1,Y_2,\ldots,Y_k$ be independent random variables  identically distributed according to the Cauchy-type distribution described in Proposition \ref{prop:cauchy}.  Assume also that $Y_1,\ldots,Y_k$ are independent of $K_n$.  We are going to show that, as $n\to\infty$,
$$
n^k\P\left(Y_1,\ldots,Y_k\notin K_n\right)=O(1).
$$
Note that the left-hand side can be written as
$$
n^k\P\left(Y_1,\ldots,Y_k\notin K_n\right)=n^k\E \left(\P^k \left(Y_1\notin K_n|K_n\right)\right)=\E \left(\frac{2n}{\omega_{d+1}}\int_{\R^d\backslash K_n}\frac{\dd x}{(1+\|x\|)^{\frac{d+1}{2}}}\right)^k.
$$
It suffices to show that
$$
\E\left[ \left(\frac{2n}{\omega_{d+1}}\int_{\R^d\backslash K_n}\frac{\dd x}{(1+\|x\|)^{\frac{d+1}{2}}}\right)^k \ind_{\{0\in K_n\}}\right] = O(1),
$$
as $n\to\infty$, because $\P(0\notin K_n) = O(\eee^{-c n})$ by Lemma~\ref{lem:NonabsorbtionCauchyPolytope_sphere}, below, with $r\downarrow 0$.
To bound the latter integral introduce the random variable
$$
\theta_n:=\min_{x\in\partial K_n}\|x\|
$$
and note that
$$
\E \left[\left(\frac{2n}{\omega_{d+1}}\int_{\R^d\backslash K_n}\frac{\dd x}{(1+\|x\|^2)^{\frac{d+1}{2}}}\right)^k \ind_{\{0\in K_n\}}\right]
\leq
\E \left(\frac{2n}{\omega_{d+1}}\int_{\R^d\backslash B_{\theta_n}(0)}\frac{\dd x}{(1+\|x\|^2)^{\frac{d+1}{2}}}\right)^k,
$$
where $B_{\theta_n}(0)$ is the ball of radius $\theta_n$ centered at the origin. From now on, for the sake of brevity, any constants only depending on  $d$ and $k$ will be denoted by $c_1$, $c_2$ etc.

Passing to polar coordinates in the expression for the above expectation we obtain
$$
I(n) :=
\E \left(\frac{2n}{\omega_{d+1}}\int_{\R^d\backslash  B_{\theta_n}(0)}\frac{\dd x}{(1+\|x\|^2)^{\frac{d+1}{2}}}\right)^k
=
\E \left(c_1 n \int_{\theta_n}^{\infty}\frac{r^{d-1}\dd r}{(1+r^2)^{\frac{d+1}{2}}}\right)^k.
$$
Note that
$$
\frac{r^{d-1}}{(1+r^2)^{\frac{d+1}{2}}}\leq \frac{1}{\max\{r^2,1\}},\qquad r>0,
$$
and therefore
$$
\int_{\theta_n}^{\infty}\frac{r^{d-1}\dd r}{(1+r^2)^{\frac{d+1}{2}}}\leq \int_{\theta_n}^{\infty} \frac{\dd r}{\max\{r^2,1\}}=
\begin{cases}
2-\theta_n, & \theta_n\leq 1,\\
\frac{1}{\theta_n}, & \theta_n>1.
\end{cases}
$$
Hence,
\begin{align*}
I(n)&\leq 2^kc_1^k n^k\P(\theta_n < 1)+ c_1^k \E \left[\left(\frac{n}{\theta_n}\right)^k\ind_{\{\theta_n\geq 1\}}\right]\\
&\leq 2^kc_1^k n^k\P(K_n\not\supset B_{1}(0))+ c_1^k \int_0^{\infty}\P\left(\left(\frac{n}{\theta_n}\right)^k\ind_{\{\theta_n\geq 1\}}>x\right)\dd x\\
&=2^kc_1^k n^k\P(K_n\not\supset B_{1}(0))+ c_1^k \int_0^{n^k}\P\left(1\leq \theta_n<nx^{-1/k}\right)\dd x\\
&\leq 2^kc_1^k n^k\P(K_n\not\supset B_{1}(0))+ c_1^k \int_0^{n^k}\P\left(K_n\not\supset B_{nx^{-1/k}}(0)\right)\dd x\\
&= 2^k c_1^k n^k\P\left(\frac{K_n}{n}\not\supset B_{n^{-1}}(0)\right)+ c_1^k \int_0^{n^k}\P\left(\frac{K_n}{n}\not\supset B_{x^{-1/k}}(0)\right)\dd x.
\end{align*}
Now we apply Lemma \ref{lem:NonabsorbtionCauchyPolytope_sphere} to bound both summands to conclude that
\begin{align*}
I(n)\leq c_2n^k \exp\{-c_3n\}+ c_2 \int_0^{n^k}\exp\left\{-\frac{1}{c_4x^{-1/k}+c_5n^{-1}}\right\}\dd x.
\end{align*}
The first summand, clearly, converges to zero and it remains to show that the integral on the right-hand side is bounded by a constant not depending on $n$. If $x\leq (c_4c_5^{-1}n)^k$, then $c_4x^{-1/k}+c_5n^{-1}\leq 2c_4x^{-1/k}$ and we have
\begin{align*}
\int_0^{(c_4c_5^{-1}n)^k}\exp\left\{-\frac{1}{c_4x^{-1/k}+c_5n^{-1}}\right\}\dd x\leq \int_0^{(c_4c_5^{-1}n)^k}\exp\left\{-\frac{1}{2c_4x^{-1/k}}\right\}\dd x\\
\leq \int_0^{\infty}\exp\left\{-\frac{1}{2c_4x^{-1/k}}\right\}\dd x<\infty.
\end{align*}
On the other hand, if $x\in ((c_4c_5^{-1}n)^k, n^k]$ (provided this interval is not empty), we have
$$
\int_{(c_4c_5^{-1}n)^k}^{n^k}\exp\left\{-\frac{1}{c_4x^{-1/k}+c_5n^{-1}}\right\}\dd x\leq \int_{(c_4c_5^{-1}n)^k}^{n^k}\exp\left\{-\frac{1}{c_5 n^{-1}+c_5n^{-1}}\right\}\dd x=O(n^k e^{-n/(2c_5)}),
$$
as $n\to\infty$. This completes the proof of the moment convergence.

\vspace*{2mm}
The formula for the expectation $\E f_{k-1}(\conv \Pi_{d,1}(2))$ in Theorem~\ref{theo:ConvergenceOfExpectations} follows from the Mecke equation \eqref{eq:MeckeEquation} applied with the function $f(x_1,\ldots,x_k;\Pi)=\ind_{\{(x_1,\ldots,x_k)\in \mathcal{F}_{k-1}(\conv \Pi)\}}$. The proof of Theorem~\ref{theo:ConvergenceOfExpectations} is complete.
\hfill $\Box$

\section{Proofs: Conic intrinsic volumes}\label{proofs:conic}

In this section we prove Theorems~\ref{theo:buchta} and~\ref{thm:ConicalIntrinsicVolumes}.
First of all, we prove the relationship \eqref{eq:ConicalMeanProjectionIntVol} between the conic mean projection volumes and the conic intrinsic volumes.
\begin{lemma}\label{lem:ConicalMeanProjection}
	For $k\in\{0,1,\ldots,d\}$ and a cone $C\subset\R^{d+1}$ we have that
	$$
	w_{k+1}(C) = \sum_{i=k+1}^{d+1}v_i(C).
	$$
\end{lemma}
\begin{proof}
	We let $\cS_k$ be the space of $k$-dimensional great subspheres of $\mathbb{S}^d$, supplied with the unique rotation invariant Haar probability measure $\tau_k$. For a spherically convex set $K\subset\mathbb{S}^d$ and $S\in\cS_k$ we denote by $K|S$ the spherical projection of $K$ onto $S$, see \cite[p.\ 263]{SW08}. The spherical mean projection volume of $K$ is given by
	$$
	W_k(K) := {1\over\omega_{k+1}}\int_{\cS_k}\sigma_k(K|S)\,\tau_k(\dd S),
	$$
where $\sigma_k$ is the $k$-dimensional Lebesgue measure on $S\in \cS_k$.
	Putting $C:=\pos K$ and using the fact that $\tau_k$ is the probability distribution of $L\cap \bS^d$, where $L\in G(d+1,k+1)$ is distributed according to the Haar measure $\nu_{k+1}$, we obtain
	\begin{align*}
	W_k(K) = {1\over\omega_{k+1}}\int_{\cS_k}\sigma_k(K|S)\,\tau_k(\dd S) = \frac 1 {\kappa_{k+1}}\int_{G(d+1,k+1)}\Vol_{k+1}(P_L(C)\cap\mathbb{B}^{d+1})\,\nu_{k+1}(\dd L).
	\end{align*}
This leads to the  equality $W_k(K)=w_{k+1}(C)$. On the other hand, from \cite[p.\ 263]{SW08} we have the relationship
	$$
	W_k(K) = \sum_{i=k}^{d}v_i(K)
	$$
with the spherical intrinsic volumes $v_i(K):=v_{i+1}(C)$.  This yields the required formula for $w_{k+1}(C)$.
\end{proof}

\begin{proof}[Proof of Theorem \ref{theo:buchta}]
We shall derive formulae for the expectations of Grassmann angles and the $f$-vectors of $C_n$ and then obtain Theorem~\ref{theo:buchta} by comparing these formulae.

\vspace*{2mm}
\noindent
\textsc{Step 1.}
We are interested in the expected Grassmann angle
$$
\E h_{k+1}(C_n)={1\over 2}\P(C_n\cap L\neq\{0\}),
$$
where $L\in G(d+1,d+1-k)$ is a random subspace with distribution $\nu_{d+1-k}$, and $k\in \{1,\ldots,d\}$.

Recall that $C_n = \pos \{U_1,\ldots,U_n\}$, where $U_1,\ldots, U_n$ are i.i.d.\ random points distributed uniformly on $\bS^d_+$. Observe that $L$ can be generated as a linear hull of $d+1-k$ i.i.d.\ random points $V_1,\ldots,V_{d+1-k}$ that are distributed uniformly on $\bS^d_+$ and independent of the $U_i$'s.

Applying the mapping $\cP$ defined by \eqref{eq:DefMappingP} together with Proposition \ref{prop:cauchy} we see that
$$
\E h_{k+1}(C_n) = \frac 12 \P\big(\conv\{\cP(U_1),\ldots, \cP(U_n)\} \cap \aff\{Z_1,\ldots,Z_{d+1-k}\}\neq\varnothing\big),
$$
where $Z_1 := \cP(V_1),\ldots, Z_{d+1-k} := \cP(V_{d+1-k})$ are independent random points in $\R^d$  distributed according to the Cauchy-type distribution described in Proposition \ref{prop:cauchy}.
Thus,
\begin{multline*}
\frac 12 - \E h_{k+1}(C_n) =
\frac 12
\int_{(\R^d)^{d+1-k}}\P(\aff \{x_1,\ldots,x_{d+1-k}\}\cap \conv\{\cP(U_{1}),\ldots, \cP(U_n)\} =\varnothing)\\
\times \prod_{i=1}^{d+1-k} {(2/\omega_{d+1})\,\dd x_i\over(1+\|x_i\|^2)^{d+1\over 2}}.
\end{multline*}

\vspace*{2mm}
\noindent
\textsc{Step 2.}
Let us now derive a formula for $\E f_k(C_n)$ or, equivalently, the expected number of $(k-1)$-dimensional faces of the random polytope $K_n:=\conv\{\mathcal {P}(U_1),\ldots,\mathcal{P}(U_n)\}$.
We have
$$
\E f_k(C_n) = \E f_{k-1}(K_n) = \E\sum_{1\leq i_1<\ldots<i_k\leq n}\ind_{\{\conv\{\cP(U_{i_1}),\ldots,\cP(U_{i_k})\}\in\cF_{k-1}(K_n)\}}.
$$
Since $\cP(U_1),\ldots,\cP(U_n)$ are independent and identically distributed according to the Cauchy-type distribution described in Proposition \ref{prop:cauchy}, we have that
\begin{align*}
\E f_k(C_n) = {n\choose k}\int_{(\R^d)^k}\P(\conv\{x_1,\ldots,x_k\}\in\cF_{k-1}(K_n)\,&|\,\cP(U_{1})=x_1,\ldots,\cP(U_{k})=x_k)\\
&\times\prod_{i=1}^k {(2/\omega_{d+1})\, \dd x_i\over(1+\|x_i\|^2)^{d+1\over 2}}.
\end{align*}
 Next, observe that conditionally on $\cP(U_{1})=x_1,\ldots,\cP(U_{k})=x_k$, we have $\conv\{x_1,\ldots,x_k\}\in\cF_{k-1}(K_n)$ if and only if $\aff \{x_1,\ldots,x_k\}\cap \conv\{\cP(U_{k+1}),\ldots, \cP(U_n)\} = \varnothing$. Therefore,
\begin{multline}\label{eq:E_f_k_explicit}
\E f_{k}(C_n)
=
{n\choose k}\int_{(\R^d)^k}\P(\aff \{x_1,\ldots,x_k\}\cap \conv\{\cP(U_{k+1}),\ldots, \cP(U_n)\} =\varnothing)\\
\times \prod_{i=1}^k {(2/\omega_{d+1})\, \dd x_i\over(1+\|x_i\|^2)^{d+1\over 2}}.
\end{multline}

\vspace*{2mm}
\noindent
\textsc{Step 3.}
Comparing the formulae obtained in Steps 1 and 2, we arrive at
$$
2 \binom {n+d+1-k}{d+1-k} \Big(\frac 12 - \E h_{k+1}(C_n)\Big) = \E f_{d+1-k} (C_{n+d+1-k}),
$$
which completes the proof.
\end{proof}

\begin{proof}[Proof of Theorem \ref{thm:ConicalIntrinsicVolumes}]
We first prove the asymptotic formula for $h_{k+1}$. For $k=0$ the result is trivial since $h_1(C_n)=1/2$, so let $k\in \{1,\ldots,d\}$.  We use Theorem~\ref{theo:buchta} together with Theorem~\ref{theo:ConvergenceOfExpectations} to obtain
\begin{equation}\label{eq:h_k_lim_proof}
n^{d+1-k}\Big({1\over 2}-\E h_{k+1}(C_n)\Big) = \frac 12 n^{d+1-k} \binom {n+d+1-k} {d+1-k}^{-1} \E f_{d-1-k} (C_{n+d+1-k}) \to B_{d+1-k,d},
\end{equation}
as $n\to\infty$. To deduce the result for the conic intrinsic volumes, recall~\eqref{eq:ConicalIntVolGrassmannAngle} and note that it implies, for $\ell\in\{0,1,\ldots,d-1\}$,
$$
\E v_\ell(C_n) = \E h_\ell(C_n) - \E h_{\ell+2}(C_n) = \Big({1\over 2}-\E h_{\ell+2}(C_n)\Big) - \Big({1\over 2}-\E h_\ell(C_n)\Big).
$$
So,
\begin{align*}
\lim_{n\to\infty} n^{d-\ell}\,\E v_\ell(C_n) = \lim_{n\to\infty}n^{d-\ell}\Big({1\over 2}-\E h_{\ell+2}(C_n)\Big) -\lim_{n\to\infty} n^{d-\ell}\Big({1\over 2}-\E h_\ell(C_n)\Big).
\end{align*}
According to~\eqref{eq:h_k_lim_proof}, the first limit equals $B_{d-\ell,d}$, while the second one is $0$ (indeed, the sequence goes to $0$ like a constant multiple of $n^{-2}$, as $n\to\infty$).

Finally, the asymptotic formulae for the mean projection volumes can be deduced in a similar way from~\eqref{eq:ConicalMeanProjectionIntVol}. Namely, for all $r\in \{0,1,\ldots,d-1\}$ we have $w_{r+1}(C_n) = h_{r+1}(C_n) + h_{r+2}(C_n)$, hence
$$
\lim_{n\to\infty}n^{d-r} \big(1 - \E w_{r+1}(C_n)\big) = \lim_{n\to\infty} n^{d-r} \Big(\frac 12 - \E h_{r+1}(C_n)\Big) +  \lim_{n\to\infty} n^{d-r} \Big(\frac 12 - \E h_{r+2}(C_n)\Big).
$$
By~\eqref{eq:h_k_lim_proof}, the second limit equals $B_{d-r,d}$, whereas the first one is $0$.
\end{proof}

\begin{proof}[Proof of Theorem \ref{thm:WeakConvergenceGrassmannAngles}]
The proof of this result is similar to the one of Theorem \ref{theo:conv_of_angle} and for this reason we restrict ourselves to a sketch. Let $U_1,\ldots,U_n$ be independent and uniformly distributed random points on $\SSd_+$ which generate the random cone $C_n$, i.e., $C_n=\pos(U_1,\ldots,U_n)$. Also, let $V_1,\ldots,V_{d+1-k}$ be independent and uniformly distributed random points on $\SSd_+$, which are independent from the $U_i$'s. Since the $(d+1-k)$-dimensional linear subspace generated by $V_1,\ldots,V_{d+1-k}$ is uniformly distributed in $G(d+1,d+1-k)$, the definition of the Grassmann angle implies that
$$
h_{k+1}(C_n) = {1\over 2}\P\big(C_n\cap{\rm lin}(V_1,\ldots,V_{d+1-k})\neq\{0\}\,|\,C_n\big).
$$
Applying now the map $\cP$ from the upper half-sphere to $\R^d$, this can be rewritten as
\begin{align*}
h_{k+1}(C_n) = {1\over 2}\P\big(\conv(\cP(U_1),\ldots,\cP(U_n))\cap\aff(\cP(V_1),\ldots,\cP(V_{d+1-k})\neq\varnothing\,|\,U_1,\ldots,U_n\big).
\end{align*}
Hence, denoting by $L_n$ the convex hull of the points $\cP(U_1)/n,\ldots,\cP(U_n)/n$ we arrive at
\begin{align*}
n^{d+1-k}\Big({1\over 2}-h_{k+1}(C_n)\Big) &= {1\over 2}n^{d+1-k}\,\P\bigg(L_n\cap\aff\Big({\cP(V_1)\over n},\ldots,{\cP(V_{d+1-k})\over n}\Big)=\varnothing\,\Big|\,L_n\bigg)\\
&={1\over 2}\int_{(\R^d)^{d+1-k}}\ind_{\{L_n\cap\aff(x_1,\ldots,x_{d+1-k})=\varnothing\}}\,\prod_{i=1}^{d+1-k}p_n(x_i)\,\dint x_i\,,
\end{align*}
where $p_n(x)$ is the function defined by \eqref{eq:Densitypn}. From this point on we can argue as in the proof of Theorem \ref{theo:conv_of_angle}. Namely, $p_n(x)$ converges to the function $(2/\omega_{d+1})\|x\|^{-(d+1)}$ for all $x\in\R^d$, while the indicator function converges to $\ind_{\{\conv \Pi_{d,1}(2)\cap\aff\{x_1,\ldots,x_{d+1-k}\} = \varnothing\}}$, as $n\to\infty$, since the random polytopes $L_n$ converge weakly to $\conv \Pi_{d,1}(2)$ on the space $\cK^d$. So, a dominated convergence argument completes the proof.
\end{proof}

\section{Proofs: Functionals of the Poisson process}\label{sec:ProofsPPP}

\subsection{Invariance property}
In our proof we shall use the following projection stability. It says that the projection of a Poisson point processes with a power-law intensity measure as in \eqref{eq:density} onto a linear subspace is again a Poisson point process of the same type within this subspace.

\begin{lemma}\label{lem:projection}
Let $\gamma>0$, $c>0$ and $k\in\{1,\ldots,d-1\}$. The orthogonal projection of $\Pi_{d,\gamma}(c)$ onto any $k$-dimensional linear subspace $L$ of $\R^d$ has the same law as $\Pi_{k, \gamma}(c)$, where we identify $L$ with $\R^k$.
\end{lemma}
\begin{proof}
First suppose that $k=d-1$. By rotational symmetry we may assume that we project onto the hyperplane $\{x_1=0\}$. The intensity of the projected Poisson point process at $(0,x_2,\ldots,x_d)$ with $x_2^2+\ldots+x_d^2=a^2$ equals
$$
\frac c {\omega_{d+\gamma}} \int_{-\infty}^{+\infty} \frac{\dd x_1}{(a^2+x_1^2)^{d+\gamma\over 2}}
=
\frac c {\omega_{d+\gamma}} \int_{-\infty}^{+\infty} \frac{a \dd y}{a^{d+\gamma}  (1+y^2)^{d+\gamma\over 2}}
=
 \frac {c\,a^{1-d-\gamma}} {\omega_{d+\gamma}}\int_{-\infty}^\infty \frac{\dd y}{(1+y^2)^{d+\gamma\over 2}},
$$
where we used the change of variables $y = x_1/a$. Applying the substitution $y^2=t$ the last integral equals
$$
\int_{-\infty}^\infty \frac{\dd y}{(1+y^2)^{d+\gamma\over 2}} = \int_0^\infty {t^{-{1\over 2}}\over (1+t)^{d+\gamma\over 2}}\,\dd t = \sqrt{\pi}\,{\Gamma(\frac{d+\gamma-1}{2})\over \Gamma(\frac{d+\gamma}{2})}
$$
by definition of Euler's beta function and its relationship to the gamma function. Hence, the intensity of the projected Poisson point process is
$$
\frac {c\,a^{1-d-\gamma}} {\omega_{d+\gamma}} \sqrt \pi\,\frac{ \Gamma(\frac{d+\gamma-1}{2})}{\Gamma(\frac{d+\gamma}{2})}
=
\frac {c} {\omega_{d+\gamma-1}} \frac 1 {a^{d+\gamma-1}}
$$
by definition of $\omega_{d+\gamma}$ and $\omega_{d+\gamma-1}$. Arguing now inductively, we arrive at the desired claim.
\end{proof}

\subsection{Expected \texorpdfstring{$T$}{T}-functional: Proof of Theorem~\ref{theo:T_expect}}\label{subsec:proof_exp_T_funct}
We are now ready to prove Theorems~\ref{theo:T_expect}, \ref{theo:T_expect_symm} and Proposition~\ref{cor:IntrinsicVolumes}.
\begin{proof}[Proof of Theorem~\ref{theo:T_expect}]
To simplify the notation, we shall write $\Pi_{d,\gamma}$ for $\Pi_{d,\gamma}(c)$ in this proof and keep $c>0$ fixed. Recall that $\conv \Pi_{d,\gamma}$ denotes the convex hull of all points of the Poisson process $\Pi_{d,\gamma}$. By Corollary~\ref{cor:polytope}, $\conv \Pi_{d,\gamma}$ is almost surely a convex polytope. Also recall that
$$
T_{a,b}^{d,k}(\conv\Pi_{d,\gamma}) = \sum_{F \in \mathcal{F}_{k}(\conv \Pi_{d,\gamma})} \dist^a(F) \Vol_k^b(F).
$$
Let us denote by $\Delta_{k-1}(x_1,\ldots,x_{k})$ the $(k-1)$-dimensional volume of the simplex with vertices $x_1,\ldots,x_{k}\in\R^d$. We denote by $E= E(x_1,\ldots,x_k)\in A(d,k-1)$ the $(k-1)$-dimensional affine subspace spanned by the points $x_1,\ldots,x_{k}$. Let also $\dist(E)$ be the distance from $E$ to the origin. By the multivariate Mecke formula for Poisson point processes \eqref{eq:MeckeEquation}, we have
\begin{multline*}
\E T_{a,b}^{d,k-1}(\conv\Pi_{d,\gamma})={1\over k!}
 \int_{(\R^d)^k}\Delta_{k-1}^b (x_1,\ldots,x_k) \,\dist^a(E)\\
\times\P\left( \conv\{x_1,\ldots,x_k\}\in \mathcal F_{k-1} (\conv\tilde \Pi_{d,\gamma})\right) \prod_{i=1}^k \frac{c\,\dd x_i}{\omega_{d+\gamma}\|x_i\|^{d+\gamma}},
\end{multline*}
where $\tilde \Pi_{d,\gamma} := \Pi_{d,\gamma} + \sum_{i=1}^k \delta_{x_i}$.  Let the linear subspace $E^\bot$ be the unique orthogonal complement of $E$ and $P_{E^\bot}$ the orthogonal projection onto $E^\bot$.  Note that $P_{E^\bot}x_1 = \ldots= P_{E^\bot}x_k$.  Clearly, the simplex $\conv\{x_1,\ldots,x_k\}$ is a $(k-1)$-dimensional face of $\conv\tilde \Pi_{d,\gamma}$ if and only if $P_{E^\bot}x_1$ is not contained in $P_{E^\bot}\conv\Pi_{d,\gamma}$. Define the non-absorption probability
\begin{equation}\label{eq:AbsobptionProbability}
p_{d,\gamma}(R) := \P(R e_1\notin \conv\Pi_{d,\gamma}), \quad R>0,
\end{equation}
where $e_1$ is any vector of unit length in $\R^d$.
By Lemma~\ref{lem:projection}, $P_{E^\bot}\Pi_{d,\gamma}$ has  the same distribution as $\Pi_{d+1-k, \gamma}$, where we identify $E^\bot$ with $\R^{d+1-k}$. Hence,
\begin{equation}\label{eq:E_T_aux}
\begin{split}
\E T_{a,b}^{d,k-1}(\conv\Pi_{d,\gamma})
={1\over k!}
\int_{(\R^d)^k}
\Delta_{k-1}^b (x_1,\ldots,x_k) \, &\dist^a(E) \\
&\times p_{d+1-k,\gamma}(\dist(E)) \prod_{i=1}^k \frac{c\,\dd x_i}{\omega_{d+\gamma}\|x_i\|^{d+\gamma}}.
\end{split}
\end{equation}
Next, we use the affine Blaschke--Petkantschin formula \eqref{eq:BlaschkePetkantschin}:
\begin{multline*}
\E T_{a,b}^{d,k-1}(\conv\Pi_{d,\gamma})
=
\frac{c^k  ((k-1)!)^{d+1-k} b_{d,k-1}}{k!\,\omega_{d+\gamma}^k} \int_{A(d,k-1)} \int_{E^k}\Delta_{k-1}^{b+d+1-k}(x_1,\ldots,x_k) \\
\times \dist^a(E) \, p_{d+1-k,\gamma}(\dist(E))
\left(\prod_{i=1}^k \frac{1}{\|x_i\|^{d+\gamma}}\right)
\dd \lambda_E^k(x_1,\ldots,x_k) \mu_{k-1}(\dd E).
\end{multline*}
Since any two affine subspaces of $\R^d$ of the same dimension and with the same distance to the origin can be transformed to each other by an orthogonal transformation, we can introduce the function
\begin{equation}\label{eq:DefFunctionh}
h(\dist(E)) := \int_{E^k} \Delta_{k-1}^{b+d+1-k}(x_1,\ldots,x_k) \left(\prod_{i=1}^k \frac{1}{\|x_i\|^{d+\gamma}}\right)  \dd \lambda_E^k(x_1,\ldots,x_k),
\end{equation}
which indeed depends on $E$ only through $\dist(E)$. With this notation we arrive at
\begin{align*}
\E T_{a,b}^{d,k-1}(\conv\Pi_{d,\gamma})
&=
\frac{c^k  (k-1)!^{d+1-k} b_{d,k-1}}{k!\,\omega_{d+\gamma}^k} \\
&\qquad\times \int_{A(d,k-1)}  \dist^a(E)\, p_{d+1-k,\gamma}(\dist(E)) h(\dist(E))\, \mu_{k-1}(\dd E).
\end{align*}
Let $\beta := b + d - k +1$. We compute
\begin{align*}
h(r)
&= \int_{(\R^{k-1})^k}  \Delta^{\beta}_{k-1}(y_1,\ldots,y_k) \prod_{i=1}^k  \frac{\dd y_i}{(r^2 + \|y_i\|^2)^{d+\gamma\over 2}}\\
&= \int_{(\R^{k-1})^k}  r^{(k-1)\beta}\Delta^{\beta}_{k-1}(z_1,\ldots,z_k) \prod_{i=1}^k  \frac{r^{k-1}\dd z_i}{r^{d+\gamma}(1 + \|z_i\|^2)^{d+\gamma\over 2}}\\
&=  r^{(k-1)k - (d+\gamma)k + \beta(k-1)} \int_{(\R^{k-1})^k} \Delta^{\beta}_{k-1}(z_1,\ldots,z_k)  \prod_{i=1}^k \frac{\dd z_i}{(1 + \|z_i\|^2)^{d+\gamma\over 2}},
\end{align*}
where we have used the change of variables $y_i = rz_i$.  Thus, the function $h$ satisfies the scaling property
$$
h(r) = r^{(k-1)k - (d+\gamma)k + \beta(k-1)} h(1).
$$
To compute the value of $h(1)$, let $Z_1,\ldots,Z_k$ be independent random variables on $\R^{k-1}$ with the so-called beta$^{\prime}$-density $f(x)$ as in~\cite{beta_polytopes}, that is,
$$
f(x) = \frac {\omega_{d+1-k+\gamma}}{\omega_{d+\gamma}}(1+ \|x\|^2)^{-\frac{d+\gamma}{2}}, \quad x\in\R^{k-1}.
$$
Recall that $\Delta_{k-1}(Z_1,\ldots,Z_k)$ is the volume of the simplex with vertices $Z_1,\ldots,Z_k$.
Then, we can interpret $h(1)$ as follows
$$
h(1) =\frac {\omega_{d+\gamma}^k}{\omega_{d+1-k+\gamma}^k} \E \Delta_{k-1}^{\beta} (Z_1,\ldots,Z_k).
$$
The moments of $\Delta_{k-1}(Z_1,\ldots,Z_k)$ have been calculated by Miles~\cite[Eqn.~(74)]{miles} and we have the explicit formula
$$
\E \Delta_{k-1}^{\beta} (Z_1,\ldots,Z_k) =
{1\over ((k-1)!)^{\beta}} \frac{\Gamma\left( \left(\frac{d+1-k+\gamma}{2} \right)k-\frac{k-1}{2}\beta \right)}{\Gamma\left( \left( \frac{d - k+1-\beta+\gamma}{2} \right)k \right)} \left( \frac{\Gamma\left(\frac{d+1-k-\beta+\gamma}{2} \right)}{\Gamma\left(\frac{d+1-k+\gamma}{2} \right)} \right)^{k} \prod_{i=1}^{k-1} \frac{\Gamma\left( \frac{i+\beta}{2} \right)}{\Gamma\left( \frac{i}{2} \right)}
$$
provided that $d - k+1-\beta+\gamma>0$.
In fact, Miles stated his result for integer moments only, but it also holds for arbitrary moments as was argued in~\cite{beta_polytopes}.

Let us consider the case $k=d$. Then $\beta=b+1$ and the above formulae simplify to
\begin{equation}\label{eq:hScaling}
h(r) = r^{(b-\gamma)d - b - 1} h(1)
\end{equation}
and
\begin{equation}\label{eq:h1}
h(1)=
\Big(\frac {\omega_{d+\gamma}}{\omega_{1+\gamma}}\Big)^d
{1\over ((d-1)!)^{b+1}}
\frac{\Gamma\left(\frac{\gamma-b}{2}d + \frac{b+1}{2} \right)}{\Gamma\left(\frac{\gamma-b}{2}d \right)} \left( \frac{\Gamma\left(\frac{\gamma-b}{2} \right)}{\Gamma\left(\frac{\gamma+1}{2} \right)} \right)^{d} \prod_{i=1}^{d-1} \frac{\Gamma\left( \frac{i+b+1}{2} \right)}{\Gamma\left( \frac{i}{2} \right)}
\end{equation}
provided that $\gamma- b >0$. We also have $h(r) = +\infty$, $r>0$, if $\gamma \leq b$.  Since
$$
b_{d,d-1} = {\omega_d\over 2} ={\pi^{d/2} \over  \Gamma({d\over 2})},
$$
the above formulae yield
$$
\E T_{a,b}^{d,d-1}(\conv\Pi_{d,\gamma})
=
\frac{c^{d}  (d-1)! \omega_d }{2d!\,\omega_{d+\gamma}^d} \int_{A(d,d-1)}  \dist^a(E) \, p_{1,\gamma}(\dist(E)) h(\dist(E))\, \mu_{d-1}(\dd E).
$$
Now, recalling the definition of $p_{1,\gamma}(R)$ from \eqref{eq:AbsobptionProbability} we obtain
\begin{equation}\label{eq:Absorp1Dim}
p_{1,\gamma}(R) = \P(R \notin \conv \Pi_{1,\gamma}) = \P(\Pi_{1,\gamma}[R,\infty) = 0)
=
\eee^{- \frac c {\omega_{\gamma+1}}\int_R^\infty \frac {\dd x} {x^{\gamma+1}}  }
=
\eee^{- \frac c {\gamma \omega_{\gamma+1}} R^{-\gamma}}.
\end{equation}
Hence,
\begin{align*}
\E T_{a,b}^{d,d-1}(\conv\Pi_{d,\gamma})
&=
\frac{c^{d}  (d-1)! \omega_d }{2d!\,\omega_{d+\gamma}^d} h(1) \\
&\qquad\times\int_{A(d,d-1)}  \dist^{a-b-1 + (b-\gamma)d}(E) \, \eee^{- \frac c {\gamma \omega_{\gamma+1}} \dist^{-\gamma}(E)} \,\mu_{d-1}(\dd E).
\end{align*}
By the definition of the measure $\mu_{d-1}$, we obtain
\begin{equation}\label{eq:TabUnsymmetricProof}
\E T_{a,b}^{d,d-1}(\conv\Pi_{d,\gamma})
=
\frac{c^{d}  (d-1)! \omega_d }{d!\,\omega_{d+\gamma}^d} h(1) \int_{0}^\infty  x^{a-b-1 + (b-\gamma)d} \eee^{- \frac c {\gamma \omega_{\gamma+1}} x^{-\gamma}}\dd x.
\end{equation}
Evaluating the integral, we get
$$
\E T_{a,b}^{d,d-1}(\conv\Pi_{d,\gamma})
=
\frac{c^{d}  (d-1)! \omega_d }{d!\,\omega_{d+\gamma}^d} h(1)
\gamma^{-1} \left(\frac{c}{\gamma \omega_{\gamma+1}}\right)^{\frac{a-b+ (b-\gamma)d}{\gamma}} \Gamma \left(\frac{(\gamma-b)d + b - a}{\gamma}\right)
$$
under the condition $(\gamma-b)d + b - a>0$. Otherwise, the integral equals $+\infty$. Applying formula \eqref{eq:h1} completes the proof.
\end{proof}

\begin{proof}[Proof of Proposition~\ref{cor:IntrinsicVolumes}]
Lemma~\ref{lem:projection} implies that for any $L\in G(d,k)$, the projected random polytope $P_L \conv \Pi_{d,\gamma}$ has the same distribution as $\conv \Pi_{k,\gamma}$ if we identify $L$ with $\R^k$. Using this together with the definition of intrinsic volumes and Fubini's theorem we get
\begin{align*}
\E V_k(\conv\Pi_{d,\gamma}) &= {d\choose k}{\kappa_d\over \kappa_k\kappa_{d-k}}\E\int_{G(d,k)}\Vol_k(P_L \conv\Pi_{d,\gamma})\,\nu_k(\dd L)\\
&= {d\choose k}{\kappa_d\over \kappa_k\kappa_{d-k}}\int_{G(d,k)}\E\Vol_k(P_L \conv\Pi_{d,\gamma})\,\nu_k(\dd L)\\
&= {d\choose k}{\kappa_d\over \kappa_k\kappa_{d-k}}\E\Vol_k(\conv\Pi_{k,\gamma}),
\end{align*}
since $\nu_k$ is a probability measure. Now, Corollary~\ref{cor:volume_expect} can be used to complete the proof.
\end{proof}

\begin{proof}[Proof of Theorem~\ref{theo:T_expect_symm}]
We keep the notation $\Pi_{d,\gamma}$ for $\Pi_{d,\gamma}(c)$. Recall that $\sconv \Pi_{d,\gamma}$ denotes the convex hull of all points of the form $\pm x$, where $x$ is a point of $\Pi_{d,\gamma}$. By Corollary~\ref{cor:polytope}, $\sconv \Pi_{d,\gamma}$ is a convex polytope a.s. Its $(k-1)$-dimensional faces have the form $\conv\{\eps_1x_1,\ldots,\eps_k x_k\}$, where $x_1,\ldots,x_k$ are distinct points from $\Pi_{d,\gamma}$ and $\eps_1,\ldots,\eps_k\in  \{+1,-1\}$.  Recalling that
$$
T_{a,b}^{d,k-1}(\sconv\Pi_{d,\gamma}) = \sum_{F \in \mathcal{F}_{k-1}(\sconv \Pi_{d,\gamma})} \dist^{a}(F) \Vol_{k-1}^b(F)
$$
we can write
\begin{multline*}
\E T_{a,b}^{d,k-1}(\sconv\Pi_{d,\gamma})
=
\frac 1{k!} \; \E \sum_{(\eps_1,\ldots,\eps_k)\in \{+1,-1\}^k} \sum_{(x_1,\ldots,x_k)\in \Pi_{d,\gamma,\neq}^k} \dist^a(\aff\{\eps_1x_1,\ldots,\eps_kx_k\})\\ \times\Delta_{k-1}^b(\eps_1x_1,\ldots,\eps_k x_k) \ind_{\{\conv\{\eps_1x_1,\ldots,\eps_k x_k\} \in \mathcal F_{k-1}(\sconv\Pi_{d,\gamma})\}}.
\end{multline*}
Interchanging the expectation and the sum over $(\eps_1,\ldots,\eps_k)$ and using the Mecke formula~\eqref{eq:MeckeEquation}, we obtain
\begin{multline*}
\E T_{a,b}^{d,k-1}(\sconv\Pi_{d,\gamma})  = \frac 1{k!} \sum_{(\eps_1,\ldots,\eps_k)\in \{+1,-1\}^k} \E
\int_{(\R^d)^{k}} \dist^a(\aff\{\eps_1x_1,\ldots,\eps_kx_k\})\\
\times \Delta_{k-1}^b(\eps_1x_1,\ldots,\eps_k x_k) \ind_{\{\conv\{\eps_1x_1,\ldots,\eps_k x_k\} \in \mathcal F_{k-1}(\sconv\tilde \Pi_{d,\gamma})\}}\prod_{i=1}^k \frac{c\, \dd x_i}{\omega_{d+\gamma} \|x_i\|^{d+\gamma}},
\end{multline*}
where $\tilde \Pi_{d,\gamma} = \Pi_{d,\gamma} + \delta_{x_1}+ \ldots + \delta_{x_k}$.
Interchanging the integral and the expectation and noting that the expectation of an indicator function is the probability of the corresponding event, we get
\begin{multline*}
\E T_{a,b}^{d,k-1}(\sconv\Pi_{d,\gamma})  = \frac 1{k!} \sum_{(\eps_1,\ldots,\eps_k)\in \{+1,-1\}^k}
\int_{(\R^d)^{k}} \dist^a(\aff\{\eps_1x_1,\ldots,\eps_kx_k\})\\
\times \Delta_{k-1}^b(\eps_1x_1,\ldots,\eps_k x_k)\; \P \left(\conv\{\eps_1x_1,\ldots,\eps_k x_k\} \in \mathcal F_{k-1}(\sconv\tilde \Pi_{d,\gamma})\right)\prod_{i=1}^k \frac{c \,\dd x_i}{\omega_{d+\gamma} \|x_i\|^{d+\gamma}}.
\end{multline*}
Now observe that
$$
\sconv \tilde \Pi_{d,\gamma} = \sconv\{\Pi_{d,\gamma} + \delta_{x_1}+ \ldots + \delta_{x_k}\} = \sconv\{\Pi_{d,\gamma} + \delta_{\eps_1 x_1}+ \ldots + \delta_{\eps_k x_k}\}.
$$
Noting that the integral remains invariant under the change of variables $\eps_1x_1\mapsto x_1,\ldots, \eps_k x_k\mapsto x_k$, we arrive at
\begin{multline*}
\E T_{a,b}^{d,k-1}(\sconv\Pi_{d,\gamma})  = \frac {2^k}{k!}
\int_{(\R^d)^{k}} \dist^a(\aff\{x_1,\ldots,x_k\})\\
\times \Delta_{k-1}^b(x_1,\ldots,x_k)\; \P \left(\conv\{x_1,\ldots,x_k\} \in \mathcal F_{k-1}(\sconv\tilde \Pi_{d,\gamma})\right)\prod_{i=1}^k \frac{c \, \dd x_i}{\omega_{d+\gamma} \|y_i\|^{d+\gamma}}.
\end{multline*}
From now on we can argue exactly as in the proof of Theorem~\ref{theo:T_expect}, but an additional factor of $2^k$ appears throughout and the non-absorption probability $p_{d,\gamma}(R)$ has to be replaced by its symmetrized version
$$
q_{d,\gamma}(R) := \P(R e_1 \notin \sconv \Pi_{d,\gamma}), \quad R>0.
$$
In particular, in the special case $k=d$, we arrive at
$$
\E T_{a,b}^{d,d-1}(\sconv\Pi_{d,\gamma})
=
\frac{(2c)^{d}  (d-1)! \omega_d }{2d!\,\omega_{d+\gamma}^d} \int_{A(d,d-1)}  \dist^a(E) \, q_{1,\gamma}(\dist(E)) h(\dist(E))\, \mu_{d-1}(\dd E).
$$
The non-absorption probability can easily be calculated as follows:
\begin{align*}
q_{1,\gamma}(R) = \P(R \notin \sconv \Pi_{1,\gamma}) &= \P(\Pi_{1,\gamma}[R,\infty) = \Pi_{1,\gamma}(-\infty, -R] = 0)
\\
&= \P(\Pi_{1,\gamma}[R,\infty)  = 0)^2
=
\eee^{- \frac{2c} {\omega_{\gamma+1}}\int_R^\infty \frac {\dd x} {x^{\gamma+1}}  }
=
\eee^{- \frac {2c} {\gamma \omega_{\gamma+1}} R^{-\gamma}}.
\end{align*}
By the definition of the measure $\mu_{d-1}$, we obtain
\begin{equation}\label{eq:TabSymmeticProof}
\E T_{a,b}^{d,d-1}(\sconv\Pi_{d,\gamma})
=
\frac{(2c)^{d}  (d-1)! \omega_d }{d!\,\omega_{d+\gamma}^d} h(1) \int_{0}^\infty  x^{a-b-1 + (b-\gamma)d} \eee^{- \frac {2c} {\gamma \omega_{\gamma+1}} x^{-\gamma}}\dd x,
\end{equation}
where $h(1)$ is given by \eqref{eq:h1}. Now a comparison of \eqref{eq:TabSymmeticProof} with \eqref{eq:TabUnsymmetricProof} in the proof of Theorem \ref{theo:T_expect} completes the proof.
\end{proof}

\begin{proof}[Proof of \eqref{eq:ConstantBd}]
We compute the constant $B_{d,d}$. Using the Blaschke--Petkantschin formula \eqref{eq:BlaschkePetkantschin} with $k=d-1$ we see that
\begin{align*}
B_{d,d} = {1\over 2}\Big({2\over\omega_{d+1}}\Big)^d\,{\omega_d\over 2}(d-1)!\int_{A(d,d-1)}\int_{E^d}&\P( \conv \Pi_{d,1}(2)\cap E=\varnothing)\Delta_{d-1}(x_1,\ldots,x_d)\\
&\times\prod_{i=1}^{d}{\dd x_i\over\|x_i\|^{d+1}}\,\mu_{d-1}(\dd E).
\end{align*}
The probability has already been computed in \eqref{eq:Absorp1Dim}:
$$
\P( \conv \Pi_{d,1}(2)\cap E=\varnothing) = e^{-{1\over\pi r}}
$$
if $r>0$ denotes the distance of $E$ to the origin. Thus, using the definition \eqref{eq:h1} of $h(1)$ and the scaling relation \eqref{eq:hScaling} (with $b=0$ and $\gamma=1$), we conclude that
\begin{align*}
B_{d,d}  &= \Big({2\over\omega_{d+1}}\Big)^d{\omega_d\over 2}\,(d-1)!\,h(1)\int_{0}^\infty e^{-{1\over \pi r}}\,r^{-(d+1)}\,\dd r\\
&= \Big({2\over\omega_{d+1}}\Big)^d{\omega_d\over 2}\,(d-1)!\,h(1)\,\pi^d (d-1)!=(d-1)!{\pi^{d-{1\over 2}}\Gamma({d+1\over 2})\over\Gamma({d\over 2})}= (2\pi)^{d-1}\Gamma\Big({d+1\over 2}\Big)^2,
\end{align*}
where in the last step we have used Legendre's duplication formula. This completes the proof.
\end{proof}

\section{Auxiliary lemmas}\label{sec:AuxLemmas}

We collect here additional technical lemmas that have been used in the arguments in the previous sections.

\subsection{Upper bound on the absorption probability}
The following lemmas are needed to prepare the proof of Lemma~\ref{lem:NonabsorbtionCauchyPolytope_sphere} that estimates the probability that the convex cone $C_n$ does not contain a small circular cone around the vector $e_0 = (1,0,\ldots,0)$. Lemma~\ref{lem:NonabsorbtionCauchyPolytope_sphere} was used in the proof of Theorem~\ref{theo:ConvergenceOfExpectations}. 

\begin{lemma}\label{lem:OriginInConvexCombination}
Suppose that for each $(\varepsilon_1,\dots,\varepsilon_d) \in \{-1,+1\}^d$ a point in $\R^d$ is given whose coordinates have the same signs as $\varepsilon_1,\dots,\varepsilon_d$. Then, the convex hull of these $2^d$ points contains the origin.
\end{lemma}
\begin{proof}
We argue by induction over the dimension $d$. The claim obviously holds for $d=1$. Suppose it is true for dimension $d-1$. Then we can take $2^{d-1}$ points corresponding to $\varepsilon_1=1$ and construct a convex combination $a_+$ of these points such that all coordinates of $a_+$ vanish except the first one (which is positive). Similarly, taking $2^{d-1}$ points corresponding to $\varepsilon_1=-1$ we construct a convex combination $a_-$ with negative first coordinate and all other coordinates being $0$.   Clearly, the origin can now be written as a convex combination of these two points $a_+$ and $a_-$.
\end{proof}
\begin{lemma}\label{lem:Quadrants}
For $r\geq 0$ and $\varepsilon_2,\ldots,\varepsilon_d \in\{-1,+1\}$ define the set
$$
A_{\varepsilon_2,\dots,\varepsilon_d}(r):=\{ (z_1,\dots,z_d) \in \R^d: z_1>r,\, \varepsilon_2 z_2>0,\dots,\varepsilon_d z_d >0 \}.
$$
Suppose that for every choice of $(\varepsilon_2,\dots,\varepsilon_d)$ a point in $A_{\varepsilon_2,\dots,\varepsilon_d}(r)$ and another point in $-A_{\varepsilon_2,\dots,\varepsilon_d}(r)$ are given. Then $(r,0,\dots,0)$ can be represented as a convex combination of these points.
\end{lemma}

\begin{proof}
By Lemma \ref{lem:OriginInConvexCombination} we can take all points in $A_{\varepsilon_2,\dots,\varepsilon_d}(r)$ or all points in $-A_{\varepsilon_2,\dots,\varepsilon_d}(r)$, respectively, corresponding to all choices of $\varepsilon_2,\dots,\varepsilon_d$ and construct a convex combination of these points such that all coordinates are zero except the first one (which is larger than $r$ or smaller than $-r$, respectively). Obviously, there exists a convex combination of these two points which is equal to $(r,0,\ldots,0)$.
\end{proof}
\begin{lemma}\label{lem:absorbtionQuadrant}
Fix $\varepsilon_2,\ldots,\varepsilon_d\in\{-1,+1\}$ and let $\xi=(\xi_1,\ldots,\xi_d) \in \R^d$ be a random vector with Cauchy-type distribution as in Proposition \ref{prop:cauchy}. Then for all $r>0$ and $n\in\N$ we have
\[
\P\left( \frac{\xi}{n} \in A_{\varepsilon_2,\dots,\varepsilon_d}(r) \right) \geq \frac{1}{\pi 2^{d-1}}\frac{1}{rn+1}.
\]
\end{lemma}
\begin{proof}
Every coordinate of $\xi$ has a one-dimensional Cauchy distribution; see, e.g., Lemma 4.3(b) in~\cite{beta_polytopes}. Since the distribution of $\xi$ is the same as that of $(\pm\xi_1,\ldots,\pm\xi_d)$ for every choice of signs,
\[
\P\left( \frac{\xi}{n} \in A_{\varepsilon_2,\dots,\varepsilon_d}(r) \right) = \left( \frac{1}{2} \right)^{d-1} \P(\xi_1 > rn) = \left( \frac{1}{2} \right)^{d-1}\left(\frac{1}{2}-\frac{1}{\pi}\arctan(rn)\right) \geq \frac{1}{\pi 2^{d-1}}\frac{1}{rn+1}
\]
by the inequality $\arctan(x)\leq \frac{\pi}{2}-\frac{1}{x+1}$ which holds for all $x\geq 0$.
\end{proof}

\begin{lemma}\label{lem:NonabsorbtionCauchyPolytope}
Let $\xi^{(1)},\dots,\xi^{(n)} \in \R^d$ be independent random vectors with a Cauchy-type distribution as in Proposition \ref{prop:cauchy}. Then, there exist constants $c_1,c_2>0$ only depending on $d$ such that, for all $r>0$ and $n\in\N$,
\[
\P\bigg( re_1 \notin \conv\bigg\{  \frac{\xi^{(1)}}{n},\dots,\frac{\xi^{(n)}}{n} \bigg\} \bigg)
\leq
c_1\exp\bigg( -\frac{c_2}{r+\frac{1}{n}} \bigg).
\]
\end{lemma}
\begin{proof}
By Lemma \ref{lem:Quadrants}, if each of the $2^d$ sets of the form $A_{\varepsilon_2,\ldots,\varepsilon_d}(r)$ and $-A_{\varepsilon_2,\ldots,\varepsilon_d}(r)$ for $\varepsilon_2,\ldots,\varepsilon_d\in\{-1,+1\}$ contains an element from the set $\{\xi^{(1)}/n,\ldots,\xi^{(n)}/n\}$, then we can write $re_1$ as a convex combination of these points. Taking the complementary event we can estimate
\begin{align*}
&\P\bigg( re_1 \notin \conv\bigg\{  \frac{\xi^{(1)}}{n},\dots,\frac{\xi^{(n)}}{n} \bigg\} \bigg)\\
&\qquad\qquad\leq \P\bigg(\text{for some }(\varepsilon_1,\ldots,\varepsilon_d)\in\{-1,+1\}^d:\varepsilon_1A_{\varepsilon_2,\ldots,\varepsilon_d}(r)\cap \bigg\{  \frac{\xi^{(1)}}{n},\dots,\frac{\xi^{(n)}}{n} \bigg\}=\varnothing\bigg).
\end{align*}
Taking the union bound, using Lemma \ref{lem:absorbtionQuadrant} and writing $c_2 = \frac 1 {\pi 2^{d-1}}$ for the constant from this lemma we arrive at
\begin{multline*}
\P\bigg( re_1 \notin \conv\bigg\{  \frac{\xi^{(1)}}{n},\dots,\frac{\xi^{(n)}}{n} \bigg\} \bigg)  \leq 2^d\, \P\bigg( \bigg\{ \frac{\xi^{(1)}}{n},\dots,\frac{\xi^{(n)}}{n} \bigg\} \cap A_{+1,\dots,+1}(r) = \varnothing \bigg)\\
= 2^d \,\left( 1- \P\left( \frac{\xi^{(1)}}{n} \in A_{+1,\dots,+1}(r) \right) \right)^n\leq 2^d \,\left( 1- \frac{c_2}{rn+1} \right)^n\leq 2^d\exp\left(-\frac{c_2n}{rn+1}\right),
\end{multline*}
where the last inequality follows since $\log(1-x)\leq -x$ for $x<1$. Putting $c_1:=2^d$ completes the proof.
\end{proof}

\begin{lemma}\label{lem:NonabsorbtionCauchyPolytope_sphere}
Let $\xi^{(1)},\dots,\xi^{(n)} \in \R^d$ be as in Lemma \ref{lem:NonabsorbtionCauchyPolytope}. Then, there exist constants $c_1,c_2,c_3>0$ only depending on $d$ such that, for all $r>0$ and $n\in\N$,
\[
\P\bigg(\conv\bigg\{  \frac{\xi^{(1)}}{n},\dots,\frac{\xi^{(n)}}{n} \bigg\}\not\supset B_{r}(0) \bigg)
\leq
c_1\exp\bigg( -\frac{1}{c_2r+\frac{c_3}{n}} \bigg).
\]
\end{lemma}
\begin{proof}
Let $e_1,\ldots,e_d$ be the standard orthonormal basis of $\R^d$.
Pick a constant $C(d)$ such that for all $r>0$  the cross-polytope $\conv\{\pm rC(d)e_j,j=1,\ldots,d\}$ contains $B_r(0)$. Then
\begin{align*}
&\P\bigg(\conv\bigg\{  \frac{\xi^{(1)}}{n},\dots,\frac{\xi^{(n)}}{n} \bigg\}\not\supset B_{r}(0) \bigg)\\
&\leq \P\bigg( \eps rC(d)e_j\notin \conv\bigg\{  \frac{\xi^{(1)}}{n},\dots,\frac{\xi^{(n)}}{n} \bigg\}\text{ for some } j=1,\ldots,d \text{ and } \eps\in \{+1,-1\} \bigg)\\
&\leq 2d\, \P\bigg(rC(d)e_1\notin \conv\bigg\{  \frac{\xi^{(1)}}{n},\dots,\frac{\xi^{(n)}}{n} \bigg\}\bigg).
\end{align*}
The claim now follows from Lemma \ref{lem:NonabsorbtionCauchyPolytope} with $rC(d)$ in place of $r$.
\end{proof}

\subsection{Some properties of the uniform distribution on the half-sphere}

\begin{lemma}
Let $U:=(\xi_0,\xi_1,\ldots,\xi_d)$ be a random vector with the uniform distribution on the $d$-dimensional half-sphere $\mathbb S_+^{d}$. Then  $\xi_0$ has probability density
\begin{equation}\label{eq:first_coordinate_density}
t\mapsto\frac{2\Gamma\left(\frac{d+1}{2}\right)}{\sqrt{\pi}\Gamma\left(\frac{d}{2}\right)}(1-t^2)^{{d\over 2}-1},\qquad t\in [0,1].
\end{equation}
\end{lemma}
\begin{proof}
This follows from the slice integration formula for spheres~\cite[Corollary A.5]{AxlerEtalHamonicFunctionTheory}, according to which the distribution function of $\xi_0$ equals
\begin{align*}
2\int_{\Sd_+}\ind_{\{x_0<t\}}\, \bar \sigma(\dd x) = {2\omega_d\over\omega_{d+1}}\int_0^t(1-x^2)^{d-3\over 2}\,\dd t,\qquad t\in[0,1],
\end{align*}
where $\bar \sigma$ is the normalized spherical Lebesgue measure on $\Sd$. Differentiation with respect to $t$ and the definitions of $\omega_{d}$ and $\omega_{d+1}$ yield \eqref{eq:first_coordinate_density}. 
\end{proof}

Recall from \eqref{eq:DefMappingP} the definition of the mapping $\cP:\mathbb{S}_+^d\cap\{x_0>0\}\to\R^d$.

\begin{lemma}\label{lem:lemma_reg_var}
Let $U:=(\xi_0,\xi_1,\ldots,\xi_d)$ be a random vector with the uniform distribution on the $d$-dimensional half-sphere $\mathbb S_+^{d}$. Then the distribution of the vector $\mathcal{P}(U)=(\xi_1/\xi_0,\ldots,\xi_d/\xi_0)$ is regularly varying in $\mathbb{R}^d$ and we have the vague convergence
\begin{equation}\label{eq:reg_var}
n\mathbb{P}\big(n^{-1}\mathcal{P}(U) \in \cdot \big)\overset{\rm v}{\longrightarrow} \nu(\cdot)
\end{equation}
on $\cM_{\R^d\backslash \{0\}}$, as $n\to\infty$, where $\nu$ is a measure on $\mathbb{R}^d\backslash \{0\}$ with density \eqref{eq:density} and with $\gamma=1$ and $c=2$.
\end{lemma}
\begin{proof}

From Proposition \ref{prop:cauchy} we know that the distribution of $\mathcal{P}(U)$ is spherically symmetric in $\mathbb{R}^d$. Whence, \eqref{eq:reg_var} is equivalent to
$$
\lim_{n\to\infty}n\mathbb{P}(n^{-1}\|\mathcal{P}(U)\|>r)=\nu(\{x\in\mathbb{R}^d:\|x\|>r\})=\frac{2}{\omega_{d+1}}\int_{\{\|x\|>r\}}\frac{\dd x}{\|x\|^{d+1}}
$$
for every $r>0$. We have
\begin{align*}
n\mathbb{P}(n^{-1}\|\mathcal{P}(U)\|>r)&=n\mathbb{P}(\xi_1^2+\cdots+\xi_d^2>n^2r^2\xi_0^2)=n\mathbb{P}(1-\xi_0^2>n^2r^2\xi_0^2)\\
&=n\mathbb{P}(\xi_0<(n^2r^2+1)^{-1/2})\longrightarrow \frac{2\Gamma\left(\frac{d+1}{2}\right)}{\sqrt{\pi}\Gamma\left(\frac{d}{2}\right)}{1\over r},
\end{align*}
as $n\to\infty$, having utilized formula \eqref{eq:first_coordinate_density} in the last passage. It remains to verify that
$$
\frac{2}{\omega_{d+1}}\int_{\{\|x\|>r\}}\frac{\dd x}{\|x\|^{d+1}}=\frac{2\Gamma\left(\frac{d+1}{2}\right)}{\sqrt{\pi}\Gamma\left(\frac{d}{2}\right)}{1\over r}.
$$
This is done by transformation into spherical coordinates:
\begin{align*}
\frac{2}{\omega_{d+1}}\int_{\{\|x\|>r\}}\frac{\dd x}{\|x\|^{d+1}} = {2\omega_d\over\omega_{d+1}}\int_r^\infty {\dd s\over s^2} = {2\omega_d\over\omega_{d+1}}{1\over r}=\frac{2\Gamma\left(\frac{d+1}{2}\right)}{\sqrt{\pi}\Gamma\left(\frac{d}{2}\right)}{1\over r},
\end{align*}
where we used the definition of $\omega_d$. The proof is complete.
\end{proof}

\section*{Acknowledgement}
We would like to that the referee, whose comments helped us to improved our text.\\
The work of AM was supported by the return fellowship of the Alexander von Humboldt foundation. DT was supported by the Deutsche Forschungsgemeinschaft (DFG) via RTG 2131 {\it High-Dimensional Phenomena in Probability -- Fluctuations and Discontinuity}. ZK and CT were supported by the DFG Scientific Network {\it Cumulants, Concentration and Superconcentration}.

\addcontentsline{toc}{section}{References}
\bibliography{convex_hull_PPP}
\bibliographystyle{plainnat}

\vspace{1cm}

\footnotesize

\textsc{Zakhar Kabluchko:} Institut f\"ur Mathematische Stochastik, Westf\"alische Wilhelms-Universit\"at M\"unster\\
\textit{E-mail}: \texttt{zakhar.kabluchko@uni-muenster.de}

\bigskip

\textsc{Alexander Marynych:} Faculty of Computer Science and Cybernetics, Taras Shev\-chen\-ko National University of Kyiv\\
\textit{E-mail}: \texttt{marynych@unicyb.kiev.ua}

\bigskip

\textsc{Daniel Temesvari:} Fakult\"at f\"ur Mathematik, Ruhr-Universit\"at Bochum\\
\textit{E-mail}: \texttt{daniel.temesvari@rub.de}

\bigskip

\textsc{Christoph Th\"ale:} Fakult\"at f\"ur Mathematik, Ruhr-Universit\"at Bochum\\
\textit{E-mail}: \texttt{christoph.thaele@rub.de}

\end{document}